\documentclass{amsart}
\usepackage{amscd}
\usepackage{amssymb}
\usepackage[notcite,notref]{showkeys}

\newcommand{\Nrd}{\mathrm{Nrd}}
\newcommand{\syma}{\{\underline{a}\}}  
\newcommand{\comment}[1]{}  
\def\ie{{\it i.e.,\ }}
\def\tA{\widetilde{A}}
\def\oA{\overline{A}_0}
\def\CH{\mathrm{ch}}
\def\dim{\mathrm{dim}}
\def\deg{\mathrm{deg}}
\def\bfx{\mathbf{x}}
\def\Fix{\text{Fix}}
\def\Aut{\operatorname{Aut}}
\def\Hom{\operatorname{Hom}}
\def\Spec{\operatorname{Spec}}
\def\lra{\longrightarrow}
\def\map#1{\ {\buildrel #1 \over \lra}\ }
\def\smap#1{\ {\buildrel #1 \over \to}\ }
\def\oo{\otimes}
\newcommand{\Symp}{\mathrm{Sym}^p}

\newcommand{\A}{\mathbb{A}}
\newcommand{\C}{\mathbb{C}}
\newcommand{\Z}{\mathbb{Z}}
\newcommand{\bL}{\mathbb{L}}
\newcommand{\bLbox}{\bL(1,\dots,1)}
\newcommand{\bP}{\mathbb{P}}
\newcommand{\bF}{\mathbb{F}}
\newcommand{\bV}{\mathbb{V}}
\newcommand{\cA}{\mathcal{A}}
\newcommand{\cB}{\mathcal{B}}
\newcommand{\cK}{\mathcal{K}}
\newcommand{\cE}{\mathcal{E}}
\newcommand{\cO}{\mathcal{O}}

\newcommand{\checkE}{\cE\,\check{}\,}
\newcommand{\barT}{\bar{Y}}
\newcommand{\RQ}{RQ}
\newcommand{\tRQ}{R\tilde{Q}}
\newcommand{\Rq}{Rq}

\def\Kweb#1{   
http:\linebreak[3]//www.\linebreak[3]math.\linebreak[3]uiuc.
\linebreak[3]edu/\linebreak[3]{K-theory/#1/}}

\input xypic
\xyoption{all} \CompileMatrices

\numberwithin{equation}{section}

\theoremstyle{plain} 
\newtheorem{thm}[equation]{Theorem}
\newtheorem{cor}[equation]{Corollary}
\newtheorem{lem}[equation]{Lemma}
\newtheorem{prop}[equation]{Proposition}

\newtheorem*{thm*}{Theorem \ref{Theorem5}}

\theoremstyle{definition}
\newtheorem{defn}[equation]{Definition}

\newtheorem*{CLproof}{Proof of the Chain Lemma \ref{thm:chainlemma} for $n=2$}
\newtheorem*{NPproof}{Proof of the Norm Principle (Theorem \ref{normprin})}

\theoremstyle{remark}
\newtheorem{rem}[equation]{Remark}
\newtheorem{ex}[equation]{Example}

\newtheorem{notate}[equation]{Notation}
\newtheorem{subrem}{Remark}[equation] 
\newtheorem{subex}[subrem]{Example} 
\newtheorem*{exx}{Example}

\begin{document}
\bibliographystyle{plain}

\title{Norm Varieties and the Chain Lemma
\break (after Markus Rost)}

\date{\today}

\maketitle
\vspace{-12pt}
\hspace{\stretch{1}}{\large\it Notes by Christian Haesemeyer
and Chuck Weibel.}\hspace{\stretch{1}}

\bigskip

The goal of this paper is to present proofs of two results of Markus Rost,
the Chain Lemma \ref{thm:chainlemma} and the Norm Principle \ref{normprin}.
These are the steps needed to complete the published verification of the 
{\it Bloch-Kato conjecture}, that the norm residue maps are isomorphisms
$K^M_n(k)/p \smap{\simeq} H^n_{et}(k,\Z/p)$
for every prime $p$, every $n$ and every field $k$ containing $1/p$.

Throughout this paper, $p$ is a fixed odd prime, and $k$ is a field of 
characteristic $0$, containing the $p$-th roots of unity. We fix an integer 
$n\ge2$ and an $n$-tuple $(a_1,...,a_n)$ of units in $k$, such that the
symbol $\syma$ is nontrivial in the Milnor $K$-group $K^M_n(k)/p$.

Associated to this data are several notions.
A field $F$ over $k$ is a {\it splitting field} for $\syma$ if $\syma_F=0$ 
in $K^M_n(F)/p$. A variety $X$ over $k$ is called a {\it splitting variety}
if its function field is a splitting field; $X$ is {\it $p$-generic} if
any splitting field $F$ has a finite extension $E/F$ of degree
prime to $p$ with $X(E)\ne\emptyset$. A {\it Norm variety} for $\syma$
is a smooth projective $p$-generic splitting variety for $\syma$ 
of dimension $p^{n-1}\!-\!1$.

The following sequence of theorems reduces the Bloch-Kato conjecture to 
the Chain Lemma \ref{thm:chainlemma} and the Norm Principle \ref{normprin};
the notion of a {\it Rost variety} is defined in \ref{def:Rostvar} below;
the definition of a {\it Rost motive} is given in \cite{Waxioms} and 
\cite{Wpatch}, and will not be needed in this paper.
\begin{list}
{}{\setlength{\leftmargin}{2.5em}}
\item[(0)] The Chain Lemma \ref{thm:chainlemma} and the Norm Principle \ref{normprin}
hold; this is proven here.
\item[(1)] Given (0), Rost varieties exist; this is Theorem \ref{thm:normvar} below,
and is proven in \cite[p.\ 253]{SJ}.
\item[(2)] 
If Rost varieties exist then Rost motives exist; 
this is proven in \cite{Wpatch}.
\item[(3)] If Rost motives exist then Bloch-Kato is true; 
this is proven in \cite{V03} and \cite{Waxioms}.
\end{list}
\noindent
Here is the statement of the Chain Lemma, 
which we quote from \cite[5.1]{SJ} and prove in 
\S\ref{sec:ModelforMoves}. A field is {\it $p$-special} 
if $p$ divides the order of every finite field extension.

\begin{thm}[Rost's Chain Lemma] \label{thm:chainlemma} 
Let $\syma \in K_n^M(k)/p$ be a nontrivial symbol, 
where  $k$ is a $p$-special field. Then
there exists a smooth projective cellular variety $S/k$ and a
collection of invertible sheaves $J =
J_1,J_1',\dotsc,J_{n-1},J_{n-1}'$ equipped with nonzero $p$-forms
$\gamma = \gamma_1,\gamma'_1\dotsc,\gamma_{n-1},\gamma'_{n-1}$ 
satisfying the following conditions.
\begin{enumerate}
\item $\dim\, S = p(p^{n-1} - 1) = p^n - p$;
\item $\{a_1,\dotsc,a_n\} = 
\{a_1,\dotsc,a_{n-2},\gamma_{n-1},\gamma'_{n-1}\}\in K_n^M(k(S))/p,$
\newline
$\{a_1,\dotsc,a_{i-1},\gamma_i\} = 
\{a_1,\dotsc,a_{i-2},\gamma_{i-1},\gamma_{i-1}'\} \in K_i^M(k(S))/p$
for $2\leq i< n$.
In particular, 
$\{a_1,\dotsc,a_{n}\}\!=\!\{\gamma,\gamma_1',\dotsc,\gamma_{n-1}'\}
\!\in\! K_n^M(k(S))/p;$
\item $\gamma\notin\Gamma(S,J)^{\otimes (-p)}$, as is evident from (2);
\item for any $s\in V(\gamma_i)\cup V(\gamma_i')$, 
the field $k(s)$ splits $\{a_1,\dotsc,a_n\};$ 
\item $I(V(\gamma_i))+I(V(\gamma_i')) \subseteq p\Z$ for all $i$,
as follows from (4);
\item $\deg(c_1(J)^{\dim\, S})$ is relatively prime to $p$.
\end{enumerate}
\end{thm}

Rost's Norm Principle concerns the group $\oA(X,\cK_1)$, which we now define.

\begin{defn}\label{def:A0K1} (Rost, \cite{R-AXK})
For any regular scheme $X$, the group 
$A_0(X,\cK_1)$ is defined to be the group generated by symbols
$[x,\alpha]$, where $x$ is a closed point of $X$ and $\alpha\in k(x)^\times$,
modulo the relations (i) $[x,\alpha][x,\alpha']=[x,\alpha\alpha']$
and (ii) for every point $y$ of dimension~1 the image of the tame symbol
$K_2(k(y))\to\oplus k(x)^\times$ is zero. 

The functor $A_0(X,\cK_1)$ is covariant in $X$ for proper maps,
because it is isomorphic to the motivic homology group 
$H_{-1,-1}(X)=\Hom_{DM}(\Z,M(X)(1)[1])$
(see \cite[1.1]{SJ}). It is also the $K$-cohomology group
$H^d(X,\cK_{d+1})$, where $d=\dim(X)$.

The reduced group $\oA(X,\cK_1)$ is defined to be the quotient of 
$A_0(X,\cK_1)$ by the difference of the two projections from 
$A_0(X\times X,\cK_1)$. As observed in \cite[1.2]{SJ}, there is a 
well defined map $N:\oA(X,\cK_1)\to k^\times$ sending $[x,\alpha]$
to the norm of $\alpha$.
\end{defn}

\begin{thm}[Norm Principle] \label{normprin}
Suppose that $k$ is a $p$-special field and that $X$ is a Norm variety
for some nontrivial symbol $\syma$. Let
$[z,\beta]\in \oA(X,\cK_1)$ be such that $[k(z):k] = p^\nu$
for $\nu>1$. Then there exists a point $x\in X$ with $[k(x):k]=p$
and $\alpha\in k(x)^\times$ such that $[z,\beta] = [x,\alpha]$
 in $\oA(X,\cK_1)$. 
\end{thm}

We will prove the Norm Principle \ref{normprin} in section
\ref{sec:NP} below.

Our proofs of these two results are based on 1998 Rost's preprint
\cite{Rost-CSV}, his web site \cite{R-CL} and Rost's lectures 
\cite{RostNotes} in 1999-2000 and 2005. The idea for writing up these
notes in publishable form originated during his 2005 course, and was reinvigorated by
conversations with Markus Rost at the Abel Symposium 2007 in Oslo. As usual, all mistakes in this paper are the responsibility of the authors.

\subsection*{Rost varieties}

In the rest of this introduction, we explain how \ref{thm:chainlemma} 
and \ref{normprin} imply the problematic Theorem \ref{thm:normvar},
and hence complete the proof of the Bloch-Kato conjecture.
We first recall the notions of a $\nu_i$-variety and a Rost variety.

Let $X$ be a smooth projective variety of dimension $d>0$.
Recall from \cite[\S16]{CC} that there is a characteristic class $s_d: K_0(X)\to\Z$ corresponding to the
symmetric polynomial $\sum t_j^d$ in the Chern roots $t_j$ of a bundle;
we write $s_d(X)$ for $s_d$ of the tangent bundle $T_X$. 
When $d=p^{\nu}-1$, we know that $s_d(X)\equiv0\pmod p$; 
see \cite[16.6 and 16-E]{CC} and \cite[pp.\,128--9]{Stong} or
\cite[II.7]{Adams}.

\begin{defn}\label{def:nu-var} 
(see \cite[1.20]{SJ}) A {\it $\nu_{n-1}$-variety} over a field
$k$ is a smooth projective variety $X$ of dimension $d=p^{n-1}-1$,
with $s_d(X)\not\equiv0\pmod{p^2}$.
\end{defn} 

For example, $s_d(\bP^d)=d+1$ by \cite[16.6]{CC}. 
Thus the projective space $\bP^{p-1}$ is a $\nu_1$-variety, and so is
any Brauer-Severi variety of dimension $p-1$. In Section \ref{sec:bpatheorem}, we will show that the bundle $\bP(\cA)$ over $S$ is a $\nu_n$-variety.

\begin{defn}\label{def:Rostvar}
A {\it Rost variety} for a sequence $\syma=(a_1,...,a_n)$ of 
units in $k$ is a $\nu_{n-1}$-variety such that:
$\{ a_1,...,a_n\}$ vanishes in $K^M_n(k(X))/p$; for
each $i<n$ there is a $\nu_i$-variety mapping to $X$;
and the motivic homology sequence
\begin{equation}\label{6.3}
\begin{CD}
H_{-1,-1}(X\times X) @>{\pi_0^*-\pi_1^*}>> H_{-1,-1}(X)
\to H_{-1,-1}(k)\quad (=k^\times).
\end{CD}\end{equation}
is exact. Part of Theorem \ref{thm:normvar} states that Rost varieties exist for every
$\syma$. 
\end{defn}

\begin{subrem}
Rost originally defined a {\it Norm Variety} for $\syma$ to be a projective splitting variety of dimension $p^{n-1}$ which 
is a $\nu_{n-1}$-variety.  (See \cite[10/20/99]{RostNotes}.)
Theorem \ref{thm:normvar}(2) says that our definition agrees with Rost's when
$k$ is $p$-special.  
\end{subrem}

Here is the statement of Theorem \ref{thm:normvar}, quoted 
from \cite[1.21]{SJ}. It assumes that the Bloch-Kato conjecture 
holds for $n-1$. 

\begin{thm}\label{thm:normvar}
Let $n\geq 2$ and $0\ne\syma = \{a_1,\dotsc,a_n\} \in K_n^M(k)/p.$ Then:

0) There exists a geometrically irreducible Norm variety for $\syma$.

\noindent
Assume further that $k$ is $p$-special. If $X$ is a Norm variety for $\syma$, then:

1) $X$ is geometrically irreducible. 

2) $X$ is a $\nu_{n-1}$-variety. 

3) each element of $\oA(X,\cK_1)$ is of the form
$[x,\alpha]$, where $x\in X$ is a closed point of degree $p$ and
$\alpha\in k(x)^\times$. 
\end{thm}

The construction of geometrically irreducible Norm varieties was 
carried out in \cite[pp.\ 254--256]{SJ}; this proves part (0) of 
Theorem \ref{thm:normvar}. Part (1) was proven in \cite[5.4]{SJ}.
Part (2) was proven in \cite[5.2]{SJ}, assuming
Rost's Chain Lemma (see \ref{thm:chainlemma}), and part (3) was proven in 
\cite[p.\ 271]{SJ}, assuming not only the Chain Lemma 
but also the Norm Principle (see \ref{normprin} below).

As stated in the introduction of \cite{SJ}, the construction of Norm varieties and the proof of Theorem \ref{thm:normvar} are part of an inductive proof of the Bloch-Kato conjecture. We point out that in the present paper, the inductive assumption (that the Bloch-Kato conjecture for $n-1$ holds) is never used. It only appears in \cite{SJ} to prove that the candidates for norm varieties constructed there are $p$-generic splitting varieties. (However, the Norm Principle \ref{normprin} is itself a statement about norm varieties.) In particular, the Chain Lemma \ref{thm:chainlemma} holds in all degrees independently of the Bloch-Kato conjecture.

\section{Forms on vector bundles}

We begin with a presentation of some well known facts about $p$-forms.

If $V$ is a vector space over a field $k$, a {\it $p$-form} on $V$ is
a symmetric $p$-linear function on $V$, \ie a linear map
$\phi:\Symp(V)\to k$. It determines a {\it $p$-ary form}, \ie a function
$\varphi: V\to k$ satisfying $\varphi(\lambda v)=\lambda^p\varphi(v)$,
by $\varphi(v) = \phi(v,v,\dots,v)$.
If $p!$ is invertible in $k$, $p$-linear forms are in 1--1 correspondence
with $p$-ary forms. 

If $V=k$ then every $p$-form may be written as $\varphi(\lambda)=a\lambda^p$
or $\phi(\lambda_1,\dots)=a\prod\lambda_i$ for some $a\in k$. Up to
isometry, non-zero 1-dimensional $p$-forms are in 1--1 correspondence with 
elements of $k^\times/k^{\times p}$. Therefore an $n$-tuple of forms
$\varphi_i$ determine a well-defined element of $K^M_n(k)/p$ which we
write as $\{\varphi_1,\dots,\varphi_n\}$.

Of course the notion of a $p$-form on a projective module over a commutative
ring makes sense, but it is a special case of $p$-forms on locally free 
modules (algebraic vector bundles), which we now define. 

\begin{defn}\label{def:pform}
If $\cE$ is a locally free $\cO_X$-module over a scheme $X$ then a $p$-form
on $\cE$ is a symmetric $p$-linear function on $\cE$, \ie a linear map
$\phi:\Symp(\cE)\to \cO_X$. If $\cE$ is invertible, we will sometimes identify
the $p$-form with the diagonal $p$-ary form
$\varphi=\phi\circ\Delta:\cE\to\cO_X$;
locally, if $v$ is a section generating $\cE$ then the form is determined
by $a=\varphi(v)$: $\varphi(t v)=a\,t^p$.
\end{defn}

\begin{subrem}
The geometric vector bundle over a scheme $X$ whose sheaf of sections is
$\cE$ is $\bV=\mathbf{Spec}(S^*(\checkE))$, where $\checkE$
is the dual $\cO_X$-module of $\cE$.
We will sometimes describe $p$-forms in terms of $\bV$.
\end{subrem}

The projective space bundle associated to $\cE$ is 
$\pi: \bP(\cE)=\mathbf{Proj}(S^*)\to X$, $S^*=S^*(\checkE)$. The tautological 
line bundle on $\bP(\cE)$ is $\bL=\mathbf{Spec}(\textrm{Sym}\,\cO(1))$,
and its sheaf of sections is $\cO(-1)$. 
The multiplication $S^*\oo\checkE\to S^*(1)$ in the symmetric algebra
induces a surjection of locally free sheaves $\pi^*(\checkE)\to\cO(1)$
and hence an injection $\cO(-1)\to\pi^*(\cE)$;
this yields a canonical morphism $\bL\to \pi^*(\bV)$ of the
associated geometric vector bundles.

\begin{defn}\label{def:tautform}
Any $p$-form $\psi:\Symp(\cE)\to\cO_X$ on $\cE$ induces a 
canonical $p$-form $\epsilon$ on the tautological line bundle $\bL$:
\[
\epsilon: \cO(-p)=\Symp(\cO(-1)) \to \Symp(\pi^*\cE) = 
	\pi^*\Symp(\cE) \map{\psi} \pi^*\cO_X =\cO_{\bP(\cE)}.
\]
\end{defn}

We will use the following notational shorthand.
For a scheme $Z$, a point $q$ on some $Z$-scheme and a
vector bundle $V$ on $Z$ we write $V|_q$ for the fiber of $V$
at $q$, {\it i.e.,} the $k(q)$ vector space $q^*(V)$ for $q\to Z$.
If $\varphi$ is a $p$-form on a line bundle $L$, $0\ne u\in L|_q$ and
$a=\varphi|_q(u^p)$, then $\varphi|_q:(L|_q)^p\to k(q)$ is the $p$-form 
$\varphi|_q(tu^p)=at^p$.

\begin{ex}\label{P(O+K)}
Given an invertible sheaf $L$ on $X$, and a $p$-form $\varphi$ on $L$,
the bundle $V=\cO\oplus L$ has the $p$-form $\psi(t,u)=t^p-\varphi(u)$.
Then $\bP(V)\to X$ is a $\bP^1$-bundle, and its tautological line 
bundle $\bL$ has the $p$-form $\epsilon$ described in \ref{def:tautform}. 

Over a point in $\bP(V)$ of the form $\infty=(0:u)$, the $p$-form on 
$\bL|_\infty$ is $\epsilon(0,\lambda u)=-\lambda^p\varphi(u)$.
If $q=(1:u)$ is any other point on $\bP(V)$
then the 1-dimensional subspace $\bL|_q$ of the vector space $V|_q$ 
is generated by $v=(1,u)$
and the $p$-form $\epsilon|_q$ on $\bL|_q$ is determined by 
$\epsilon(v)=\psi(1,u)=1-\varphi(u)$ in the sense that
$\epsilon(\lambda\,v) = \lambda^p(1-\varphi(u))$.
\end{ex}

One application of these ideas is the formation of the sheaf of
Kummer algebras associated to a $p$-form. Recall that if $L$ is a line
bundle then the $(p\!-\!1)$st symmetric power of $\bP(\cO\oplus L)$ is
$\mathrm{Sym}^{p-1}\bP(\cO\oplus L) = \bP(\cA(L))$, where
$\cA(L)=\bigoplus_{i = 0}^{p-1} L^{\otimes i}$.

\begin{defn}\label{Kummeralgebra}
If $L$ is a line bundle on $X$, equipped with a $p$-form $\phi$, the {\it Kummer algebra} $\cA_\phi(L)$ 
is the vector bundle $\cA(L)=\bigoplus_{i = 0}^{p-1} L^{\otimes i}$ 
regarded as a bundle of algebras as in \cite[3.11]{SJ}; 
locally, if $u$ is a section generating $L$ then
$\cA(L)\cong \cO[u]/(u^p-\phi(u))$. If $x\in X$ and $a=\phi|_x(u)$
then the $k(x)$-algebra $\cA|_x$ is the Kummer algebra $k(x)(\root{p}\of{a})$,
which is a field if $a\not\in k(x)^p$ and $\prod k(x)$ otherwise.
\end{defn}
\goodbreak

Since the norm on $\cA_\phi(L)$ is given by a homogeneous polynomial
of degree $p$, we may regard the norm as a map from $\Symp\cA_\phi(L)$ 
to $\cO$.  The canonical $p$-form $\epsilon$ on the tautological 
line bundle $\bL$ on the projective bundle $\bP=\bP(\cA(L))$,
given in \ref{def:tautform}, agrees with the natural $p$-form:
$$ 
\bL^{\otimes p} \to \Symp \pi^*\cA(L) \map{N} \cO_{\bP},
$$
where $\pi:\bP\to X$ is the structure map and the canonical inclusion of 
$\bL$ into $\pi^*(\cA(L))=\oplus_0^{p-1}\pi^*L^{\otimes i}$ 
induces the first map. 

\vspace{.1cm}
Recall from \ref{def:tautform} and \ref{Kummeralgebra} that 
$\phi$ is a $p$-form on $L$,
$\psi=(1,-\varphi)$ is a $p$-form on $\cO\oplus L$ and $\epsilon$
is the canonical $p$-form on $\bL$ induced from $\psi$.

\begin{lem}\label{lem:pth-power}
Suppose that $x\!\in\! X$ has $\phi|_x\!\ne\!0$ and that $0\!\ne\! u\in L|_x$. 
Then $\epsilon|_{(0:u)}\ne0$. Moreover,
$\phi(u)\in k(x)^{\times p}$ iff there is a point $\ell\in\bP(\cO\oplus L)$
over $x$ so that $\epsilon|_\ell=0$.
\end{lem}

\begin{proof}
Let $w=(t,su)$ be a point of $\bL|_x$ over 
$\ell=(t:su)\in\bP(\cO\oplus L)|_x$.
If $t=0$ then $\ell=(0:u)$ and $\epsilon(w)=-s^p\phi(u)$, 
which is nonzero for $s\ne0$.
If $t\ne0$ then $\epsilon|_\ell$ is determined by
the scalar $\epsilon(w)=\psi(t,su)=t^p-s^p\phi(u)$. 
Thus $\epsilon|_\ell=0$ iff $\phi(u)=(t/s)^p$.
\end{proof}

\begin{subrem}
Here is an alternative proof, using
the Kummer algebra $K=k(x)(a)$, $a=\root{p}\of{\phi(u)}$.
Since $\epsilon(w)=\psi(t,su)$ is the norm of the nonzero element 
$t-sa$ in $K$, the norm $\epsilon(w)$ is zero iff
the Kummer algebra is split, i.e., ${\phi(u)}=a^p\in k(x)^{\times p}$.
\end{subrem}

\medskip
Finally, the notation $\{\gamma,\dots,\gamma'_{n-1}\}$ in the Chain
Lemma \ref{thm:chainlemma} is a special case of the notation in the
following definition.

\begin{defn}\label{def:symbol}
Given line bundles $H_1,\dots,H_n$ on $X$, $p$-forms $\alpha_i$
on $H_i$, and a point $x\in X$ at which each form $\alpha_i|_x$ is nonzero,
we write $\{\alpha_1,\dots,\alpha_n\}|_x$ for the element 
$\{\alpha_1|_x,\dots,\alpha_n|_x\}$ of $K^M_n(k(x))/p$ described before
\ref{def:pform}: if $u_i$ is a generator of $H_i|_x$ and 
$\alpha_i|_x(u_i)=a_i$ then 
$\{\alpha_1,\dots,\alpha_n\}|_x=\{ a_1,\dots,a_n\}$.
\end{defn}

We record the following useful consequence of this construction.
\begin{lem}\label{lem:specialize}
Suppose that the $p$-forms $\alpha_i$ are all nonzero at the generic
point $\eta$ of a smooth $X$. On the open subset $U$ of $X$ of points $x$
on which each $\alpha_i|_x\ne0$, the symbol
$\{\alpha_1|_x,\dots,\alpha_n|_x\}$ in $K^M_n(k(x))/p$ is obtained by
specialization from the symbol in $K^M_n(k(X))/p$.
\end{lem}

\section{The Chain Lemma when $n=2$.}\label{sec:n=2bis}

The goal of this section is to construct certain iterated projective bundles
together with line bundles and $p$-forms on them as
needed in the case $n=2$ of the Chain Lemma \ref{thm:chainlemma}.
Our presentation is based upon Rost's lectures \cite{RostNotes}.

We begin with a generic construction, which starts with a pair 
$K_0$, $K_{-1}$ of line bundles on a variety $X_0=X_{-1}$ and
produces a tower of varieties $X_r$, equipped with distinguished
lines bundles $K_r$. Each $X_r$ is a product of $p-1$ projective line 
bundles over $X_{r-1}$, so $X_r$ has relative dimension 
$r(p-1)$ over $X_0$.

\begin{defn}\label{def:tower} 
Given a morphism $f_{r-1}: X_{r-1}\to X_{r-2}$ and line bundles
$K_{r-1}$ on $X_{r-1}$, $K_{r-2}$ on $X_{r-2}$, we form the projective
line bundle $\bP(\cO\oplus K_{r-1})$ over $X_{r-1}$ and its
tautological line bundle $\bL$.  By definition, $X_{r}$ is the product
$\prod_1^{p-1}\bP(\cO\oplus K_{r-1})$ over $X_{r-1}$.  Writing $f_{r}$
for the projection $X_r\to X_{r-1}$, and $\bL_{r}$ for the exterior
product $\bL\boxtimes\cdots\boxtimes\bL$ on $X_r$, we define the line
bundle $K_{r}$ on $X_{r}$ to be $K_{r}=(f_{r}\circ
f_{r-1})^*(K_{r-2})\otimes\bL_r$.
\begin{equation*}
X_r \map{f_r} X_{r-1} \map{f_{r-1}} X_{r-2} \dotsm X_1 \map{f_1} X_0 = X_{-1}.
\end{equation*}
\end{defn}

\begin{ex}[$k$-tower]\label{ex:ktower}
The {\it $k$-tower} is the tower obtained when we start with
$X_0=\Spec(k)$, using the trivial line bundles $K_{-1}$, $K_0$.
Note that $X_1=\prod\bP^1$ and $K_1=\bL_1$, while $X_2$ is a product of
projective line bundles over $\prod\bP^{1}$, and $K_2=\bL_2$.
\end{ex}

In the Chain Lemma (Theorem \ref{thm:chainlemma}) for $n=2$ 
we have $S = X_p$ in the $k$-tower, and the line bundles are 
$J = J_1 = K_p$, $J'_1 = f_p^* (K_{p-1})$.  Before defining the 
$p$-forms $\gamma_1$ and $\gamma_1'$ in \ref{def:gamma2}, 
we quickly establish \ref{part6/n=2}; this verifies 
part (6) of Theorem \ref{thm:chainlemma}, 
that the degree of $c_1(K_p)^{p^2-p}$ is prime to $p$.

If $L$ is a line bundle over $X$, and $\lambda=c_1(L)$,
the Chow ring of $\bP=\bP(\cO\oplus L)$ is $CH(\bP)=CH(X)[z]/(z^2-\lambda z)$,
where $z=c_1(\bL)$. If $\pi:\bP\to X$ then $\pi_*(z)=-1$ in $CH(X)$.
Applying this observation to the construction of $X_r$ out of $X=X_{r-1}$ 
with $\lambda_{r-1}=c_1(K_{r-1})$, we have
\[  CH(X_{r})=
CH(X_{r-1})[z_{r,1},\dots,z_{r,p-1}]/
(\{ z_{r,j}^2-\lambda_{r-1} z_{r,j} ~\vert~ j=1,\dots,p-1 \}),
\]
where  $z_{r,j}$ is the first Chern class of the $j$th tautological 
line bundle $\bL$. (Formally, $CH(X_{r-1})$ is identified with a subring
of $CH(X_r)$ via the pullback of cycles.)
By induction on $r$, this yields the following result:

\begin{lem}\label{lem:CH(Xr)}
$CH^*(X_r)$ is a free $CH^*(X_0)$-module.
A basis consists of the monomials $\prod z_{i,j}^{e_{i,j}}$ for 
$e_{i,j}\in\{0,1\}$, $0<i\le r$ and $0<j<p$.
As a graded algebra, 
$ CH^*(X_r)/p \cong CH^*(X_0)/p\otimes_{R_0}R_r$, 
where $R_0=\bF_p[\lambda_0,\lambda_{-1}]$ and
\begin{gather*}
R_r = 
\bF_p[\lambda_{-1},\lambda_0,\dotsc,\lambda_r,z_{1,1},\dots,z_{r,p-1}]/I_r, \\
I_r = (
\{ z_{i,j}^2 - \lambda_{i-1}z_{i,j}~\vert~ 1\le i\le r,\, 0<j<p\},
~\{\lambda_i-\lambda_{i-2}-{\scriptstyle\sum}_{j=1}^{p-1} 
z_{i,j}~|~1\le i\le r\}).
\end{gather*}
\end{lem}

\begin{defn}\label{def:z-zeta}
For $r=1,\dots,p$, set $z_r=\sum_{j=1}^{p-1} z_{r,j}$ and 
$\zeta_r=\prod z_{r,j}$.
It follows from Lemma \ref{lem:CH(Xr)} that $\lambda_i=\lambda_{i-2}+z_i$
and $z_i^p=\sum z_{r,j}^p = \sum z_{r,j}\lambda_{r-1}^{p-1} =
z_i\lambda_{i-1}^{p-1}$ in $R_r$ and hence in $CH(X_r)/p$.
\end{defn}

By Lemma \ref{lem:CH(Xr)}, if $1\le r\le p$ then multiplication by 
$\prod \zeta_i\in CH^{r(p-1)}(X_r)$ is an isomorphism 
$CH_0(X_0)/p\map{\sim} CH_0(X_r)/p$. If $X_0=\Spec(k)$ then 
$CH_0(X_r)/p \cong\bF_p$, and is generated by $\prod \zeta_i$.

\begin{lem}\label{lem:deg(yz)}
If $y\in CH_0(X_0)$, the degree of $y\cdot\zeta_1\cdots\zeta_r$ is 
$(-1)^{r(p-1)}\deg(y)$. 
\end{lem}

\begin{proof} 
The degree on $X_r$ is the composition of the $(f_i)_*$.
The projection formula implies that $(f_r)_*(\zeta_r)=(-1)^{p-1}$, and 
\[ 
(f_r)_*(y\cdot \zeta_1\cdots \zeta_r) 
 = (y\cdot \zeta_1\cdots \zeta_{r-1})\cdot(f_r)_*(\zeta_r)
 = (-1)^{p-1}y\cdot \zeta_1\cdots \zeta_{r-1}.
\] 
Hence the result follows by induction on $r$.
\end{proof}\goodbreak

\begin{prop}\label{part6/n=2}
For every 0-cycle $y$ on $X_0$ and $1\le r\le p$, $\lambda_r=c_1(K_r)$
satisfies $y\,\lambda_r^{r(p-1)} \equiv y\,\zeta_1\cdots\zeta_r$ 
in $CH_0(X_r)/p$, and 
$\deg(y\lambda_r^{r(p-1)})\equiv\deg(y)\pmod p$.

For the $k$-tower \ref{ex:ktower} (with $y=1$), we have
$\deg(\lambda_p^{p^2-p})\equiv1\pmod p$.
\end{prop}

\begin{proof}
If $r=1$ this follows from $y\lambda_{-1}=y\lambda_0=0$ in $CH(X_0)$: 
$\lambda_1=z_1+\lambda_{-1}$ and $y\cdot\zeta_1\equiv y\,\lambda_1^{p-1}$.
For $r\ge2$, we have $\lambda_r=z_{r}+\lambda_{r-2}$
and $z_r^p=z_r\lambda_{r-1}^{p-1}$ by \ref{def:z-zeta}.  
Because $p - r \geq 0$, we have 
\begin{align*}
\lambda_r^{r(p-1)} =& (z_{r}+\lambda_{r-2})^{p(r-1) + (p-r)} \equiv 
(z_{r}^p + \lambda_{r-2}^p)^{r-1}\!\cdot (z_r+\lambda_{r-2})^{p-r} \mod p \\
=& (z_r \lambda_{r-1}^{p-1} + 
		\lambda_{r-2}^p)^{r-1} (z_r+\lambda_{r-2})^{p-r}
\equiv   
\zeta_r\,\lambda_{r-1}^{(r-1)(p-1)} + T \mod p,
\end{align*}
where $T\!\in CH(X_{r-1})[z_r]$ 
is a homogeneous polynomial of total degree $<\!p\!-\!1$ in $z_r$. 

By \ref{lem:CH(Xr)}, the coefficients of $yT$ are elements of
$CH(X_{r-1})$ of degree $>\dim(X_{r-1})$, so $yT$ must be zero.
Then by the inductive hypothesis,
\[
y\,\lambda_{r-1}^{r(p-1)} \equiv y\,\zeta_r \lambda_{r-1}^{(r-1)(p-1)} \equiv
y\,\zeta_r\cdot (\zeta_1\dotsm \zeta_{r-1})
\]
in $CH^*(X_r)/p$, as claimed. Now the degree assertion follows from 
Lemma \ref{lem:deg(yz)}.
\end{proof}

\subsection*{The $p$-forms}
We now turn to the $p$-forms in the Chain Lemma \ref{thm:chainlemma},
using the $k$-tower \ref{ex:ktower}. We will inductively equip the 
line bundles $\bL_r$ and $K_r$ of \ref{ex:ktower} with $p$-forms 
$\Psi_r$ and $\varphi_r$; the $\gamma_1$ and $\gamma_1'$ of the
Chain Lemma \ref{thm:chainlemma} will be $\varphi_p$ and $\varphi_{p-1}$.

When $r=0$, we equip the trivial line bundles $K_{-1}, K_0$ on
$X_0=\Spec(k)$ with the $p$-forms $\varphi_{-1}(t)=a_1t^p$
and $\varphi_{0}(t)=a_2t^p$. The $p$-form $\varphi_{r-1}$ on $K_{r-1}$
induces a $p$-form $\psi(t,u)=t^p-\varphi_{r-1}(u)$ on $\cO\oplus K_{r-1}$
and a $p$-form $\epsilon $ on the tautological line bundle $\bL$, 
as in Example \ref{P(O+K)}. As observed in Example \ref{P(O+K)},
at the point $q=(1:x)$ of $\bP(\cO\oplus K_{r-1})$
we have $\epsilon(y)=\psi(1,x)=1-\varphi_{r-1}(x)$.

\begin{defn}\label{def:gamma2}
The $p$-form $\Psi_r$ on $\bL_r$ is the product form $\prod\psi$:
\[ \Psi_r(y_1\boxtimes\cdots\boxtimes y_{p-1})=\prod \psi(y_i).
\]
The $p$-form $\varphi_r$ on 
$K_r = (f_{r-1}\circ f_r)^*(K_{r-2})\otimes \bL_r$ is defined to be
\[
\varphi_r = (f_{r-1}\circ f_r)^*(\varphi_{r-2})\otimes\Psi_r.
\]
\end{defn}

\begin{prop}\label{splitting/n=2} 
Let $x=(x_1,\dots,x_{p-1})\in X_r$ be a point with residue field 
$E = k(x)$. For $-1\leq i\leq r$, 
choose generators $u_i$ and $v_i$ for the one-dimensional 
$E$ vector spaces $K_i|_x$ and $\bL_i|_x$ respectively,
in such a way that $u_i = u_{i-2}\otimes v_i$.
\begin{enumerate}
\item If  $\varphi_i|_x = 0$ for some $1\le i\le r$ 
then $\{a_1,a_2\}_E = 0 \in K_2 (E)/p$.
\item If $\varphi_i|_x \neq 0$ for all $i$, $1\leq i \leq r$, then 
\[
\{a_1,a_2\}_E = (-1)^r\{\varphi_{r-1}(u_{r-1}),\varphi_r(u_r)\}\in K_2(E)/p.
\]
\end{enumerate} 
\end{prop}

\begin{proof}
By induction on $r$. Both parts are obvious if $r=0$. To prove the
first part, we may assume that $\varphi_i|_x\ne0$ for $1\le i\le r-1$, 
but $\varphi_r|_x = 0$.  We have $u_r = u_{r-2}\otimes v_r$ and 
by the definition of $\varphi_r$, we conclude that
\[ 
0 = \varphi_r(u_r) = \varphi_{r-2}(u_{r-2}) \Psi_r(v_r),
\]
whence $\Psi_r(v_r)=0$. Now the element $v_r\ne0$ is a tensor product 
of sections $w_j$ and $\Psi_r(v_r)=\prod\psi(w_j)$ so $\psi(w_j)=0$
for a nonzero section $w_j$ of $\bL|_{x_j}$.
By Lemma \ref{lem:pth-power}, $\varphi_{r-1}(u_{r-1})$ is a
$p$th power in $E$.
Consequently, $\{\varphi_{r-2}(u_{r-2}),\varphi_{r-1}(u_{r-1})\}_E = 0$ in $K_2 (E)/p$.
This symbol equals $\pm\{a_1,a_2\}_E$ in $K_2(E)/p$, by (2) and induction.
This finishes the proof of the first assertion.

For the second claim, we can assume by induction that 
\[
\{a_1,a_2\}_E = \pm \{\varphi_{r-2}(u_{r-2}),\varphi_{r-1}(u_{r-1})\}_E.
\] 
Now $\varphi_r(u_r) = \varphi_{r-2}(u_{r-2})\Psi_r(v_{r}).$
But  $\{\varphi_{r-1}(u_{r-1}),N_{\varphi_{r-1}}(v_{r-1})\} = 0$
by Lemma \ref{K2Tate} below. 
We conclude that 
\[
\{\varphi_{r-2}(u_{r-2}),\varphi_{r-1}(u_{r-1})\}_E \equiv 
-\{\varphi_{r-1}(u_{r-1}),\varphi_{r}(u_{r})\}_E \mod p;
\] 
this concludes the proof of the second assertion.
\end{proof}

\begin{lem}\label{K2Tate}
For any field $k$ any 
$a\in k^\times$ and any $b$ in $K_a = k[\sqrt[p]{a}]$, the symbol
$\{a,N_{K_a/k}(b)\}$ is trivial in $K_2 (k)/p$.
\end{lem}

\begin{proof}
Because $\{ a,b\}=p\{\root{p}\of{a},b\}$ vanishes in 
$K_2(K_{a}])/p$, we have $\{a,N(b)\}=N\{ a,b\}=pN(\{\root{p}\of{a},b\})=0$.
\end{proof}

\begin{CLproof}
We verify the conditions for the variety $S=X_p$ in the $k$-tower 
\ref{ex:ktower}; the line bundles $J = J_1 = K_p$, $J'_1 = f_p^*(K_{p-1})$;
the $p$-forms $\gamma_1$ and $\gamma'_1$ in \ref{thm:chainlemma}
are the forms $\varphi_p$ and $\varphi_{p-1}$ of \ref{def:gamma2}.
Part (1) of Theorem \ref{thm:chainlemma} is immediate from the 
construction of $S = X_p$; 
parts (2) and (4) were proven in Proposition \ref{splitting/n=2}; 
parts (3) and (5) follow from (2) and (4); and part (6) is 
Proposition \ref{part6/n=2} with $y=1$. 
\qquad\qed\end{CLproof}

\vspace*{-8pt}
\section*{Norm Principle for $n=2$}

The Norm Principle for $n=2$ was implicit in the Merkurjev-Suslin
paper \cite[4.3]{MS}. We reproduce their short proof, which uses the
the Severi-Brauer variety $X$ of the cyclic division
algebra $D=A_\zeta(a,b)$ attached to a nontrivial symbol $\{ a,b\}$
in $K_2(k)/p$ and a $p$th root of unity $\zeta$; $X$ is a Norm variety
for the symbol $\{ a,b\}$.

\begin{thm}[Norm Principle for $n=2$]\label{MSnorm}
If $x\in X$ and $[k(x):k]=p^m$ for $m>1$ then for all $\lambda\in k(x)$
there exists $x'\in X$ and $\lambda'\in k(x')$ so that $[k(x'):k]\le p$
and $[x,\lambda]=[x',\lambda']$ 
in $\oA(X,\cK_1)$.
\end{thm}

\begin{proof}
By Merkurjev-Suslin \cite[8.7.2]{MS}, $N:\oA(X,\cK_1)\to k^\times$ is
an injection with image $\Nrd(D)\subseteq k^\times$. Therefore the unit
$N([x,\lambda])$ of $k$ can be written as the reduced norm of an element
$\lambda'\in D$. The subfield $E=k(\lambda')$ of $D$ has degree $\le p$, 
and corresponds to a point $x'\in X$. 
Since $N([x',\lambda'])=\Nrd(\lambda')=N([x,\lambda])$,
we have $[x,\lambda]=[x',\lambda']$ in $\oA(X,\cK_1)$. 
\end{proof}

\section{The Symbol Chain}\label{sec:SymbolChain}

Here is the pattern of the chain lemma in all weights. 

We start with a sequence $a_1, a_2,\dots$ of units of $k$, and 
the function $\Phi_0(t)=t^p$.
For $r\ge1$, we inductively define functions $\Phi_r$ in $p^r$ variables
and  $\Psi_r$ in $p^r-p^{r-1}$ variables, taking values in $k$, and prove
(in \ref{lem:basicfact}) that $\{a_1,...,a_r,\Phi_r(\bfx)\}\equiv0\pmod p$.
Note that $\Phi_r$ and $\Psi_r$ depend only upon the units $a_1,...,a_r$.
We write $\bfx_i$ for a sequence of $p^r$ variables $x_{ij}$ (where
$j=(j_1,\dots,j_r)$ and $0\le j_t<p$),  and we inductively define
\begin{align}\label{eq:Psi}
\Psi_{r+1}(\bfx_1,...,\bfx_{p-1})=&\prod\nolimits_{i=1}^{p-1}
	 \bigl[ 1-a_{r+1}\Phi_r(\bfx_i)\bigr], \\
\Phi_{r+1}(\bfx_0,...,\bfx_{p-1}) =& \label{eq:Phi}
	\Phi_r(\bfx_0)\Psi_{r+1}(\bfx_1,...,\bfx_{p-1}).
\end{align}

We say that two rational functions are {\it birationally equivalent} 
if they can be transformed into one another by an automorphism (over 
the base field $k$) of the field of rational functions.

\begin{ex}\label{ex:r=1}
$\Psi_1(x_1,...,x_{p-1})=\prod(1-a_1x_i^p)$ and
$\Phi_1(x_0,...,x_{p-1})$ is $x_0^p\prod(1-a_1x_i^p)$,
the norm of the element $x_0\prod(1-x_i\alpha_1)$ in the 
Kummer extension $k(\bfx)(\alpha_1)$, $\alpha_1=\sqrt[p]{a_1}$. 
Thus $\Phi_1$ is birationally equivalent to symmetrizing in the $x_i$,
followed by the norm from $k[\sqrt[p]{a_1}]$ to $k$.
More generally, $\Psi_r(\bfx_1,...,\bfx_{p-1})$ is the norm of
an element in $k(\bfx_1,...,\bfx_{p-1})(\sqrt[p]{a_r})$.
\end{ex}

\begin{subex}\label{ex:Weilrestrict}
It is useful to interpret the map $\Phi_1$ geometrically.
Let $R_{k(\alpha)/k}\A^1$ denote the variety, isomorphic to $\A^p$,
which is the Weil restriction (\cite{Weil}) of the affine line over 
$k(\alpha)$, so that there is a morphism
$N:R_{k(\alpha)/k}\A^1\to\A^1$ corresponding to the norm map.
The function $k^p\to k(\alpha)$ defined by 
$$(x_0,s_1,\dots,s_{p-1}) \mapsto 
   x_0(1-s_1\alpha+s_2\alpha^2-\cdots\pm s_{p-1}\alpha^{p-1})
$$
induces a birational map $\A^p\map{m} R_{k(\alpha)/k}\A^1$.
Finally, let $q:\A^{p-1}\to \A^{p-1}/\Sigma_{p-1}\cong\A^{p-1}$ be the
symmetrizing map sending $(x_1,\dots)$ to the elementary symmetric functions
$(s_1,\dots)$. Then the following diagram commutes:
\begin{equation*}
\xymatrix{
{\A}^p=\A^1\times\A^{p-1}\ar[r]^{1\times q}\ar[drr]_{\Phi_1} 
& {\A}^1\times\A^{p-1} \ar[r]_{\text{birat.}}^{m~}
& R_{k(\alpha)/k}\A^1 = \A^p \ar[d]^N \\ 
&& \quad{\A}^1.
} \end{equation*}
\end{subex}

\begin{subrem}\label{rem:p=2}
If $p=2$, $\Phi_1(x_0,x_1)=x_0^2(1-a_1x_1^2)$ is birationally equivalent
to the norm form $u^2-a_1v^2$ for $k(\sqrt{a_1})/k$, and
$\Phi_2=\Phi_1(\bfx_0)[1-a_2\Phi_1(\bfx_1)]$ is birationally equivalent
to the norm form $\langle\langle a_1,a_2\rangle\rangle =
(u^2-a_1v^2)[1-a_2(w^2-a_1t^2)]$ for the quaternionic algebra 
$A_{-1}(a_1,a_2)$.

More generally, $\Phi_n$ is birationally equivalent to the Pfister form 
$$\langle\langle a_1,...,a_r\rangle\rangle
 = \langle\langle a_1,...,a_{r-1}\rangle\rangle \perp
a_n\langle\langle a_1,...,a_{r-1}\rangle\rangle
$$
and $\Psi_r$ is equivalent to the restriction of the Pfister form
to the subspace defined by the equations $\bfx_0=(1,\dots,1)$.
\end{subrem}

\begin{subrem}[Rost]\label{rem:p=3}
Suppose that $p=3$. Then 
$\Phi_2$ is birationally equivalent to (symmetrizing, followed by) the
reduced norm of the algebra $A_{\zeta}(a_1,a_2)$ and $\Phi_3$ is
equivalent to the norm form of the exceptional Jordan algebra
$J(a_1,a_2,a_3)$.  When $r=4$, Rost showed that the set of nonzero values 
of $\Phi_4$ is a subgroup of $k^\times$.
\end{subrem}

For the next lemma, it is useful to introduce the function field $F_r$
in the $p^r$ variables $x_{j_1,\dots,j_r}$, $0\le j_t<p$. Note that
$F_r$ is isomorphic to the tensor product of $p$ copies of $F_{r-1}$.

\begin{lem}\label{lem:basicfact}
$\{ a_1,...,a_r,\Phi_r(\bfx)\}=
\{ a_1,...,a_r,\Psi_r(\bfx)\} = 0\in K_{r+1}^M(F_r)/p$. 

If $b\in k$ is a nonzero value of $\Phi_r$, then 
$\{ a_1,...,a_r,b\} = 0\in K_{r+1}^M(k)/p$. 
\end{lem}

\begin{proof}
By Lemma \ref{K2Tate}, $\{ a_r,\Psi_r(\bfx)\}=0$ because $\Psi_r(\bfx)$
is a norm of an element of $k(\bfx)(\alpha_r)$ by \ref{ex:r=1}.
If $r=1$ then $\{ a_1,\Phi_1(\bfx)\}=\{ a_1,x_0^p\}\equiv0$ as well.
The result for $F_r$ follows by induction: 
\[ \{ a_1,...,a_{r+1},\Phi_{r+1}(\bfx)\} = 
  \{ a_1,...,a_{r+1},\Phi_r(\bfx_0)\} 
  \{ a_1,...,a_{r+1},\Psi_{r+1}(\bfx)\} = 0.
\]
The result for $b$ follows from the first assertion, and specialization
from $F_r$ to $k$.
\end{proof}

\begin{rem}\label{exist:NV}
For any value $b\in k^\times$ of $\Phi_n$, any desingularization $X$ 
of the projective closure of the affine hypersurface 
$X_b=\{\bfx: \Phi_n(\bfx)=b \}$ will be a Norm variety for the symbol 
$\{ a_1,...,a_{n},b \}$ in $K^M_{n+1}(k)/p$.

Indeed, since $\dim(X_b)=p^n-1$, we see from Lemma \ref{lem:basicfact} that
every affine point of $X_b$ splits the symbol. In particular, the generic
point of $X_b$ is a splitting field for this symbol.
By specialization, every point of $X_b$ and $X$ splits the symbol.

The symmetric group $\Sigma_{p-1}$ acts on $\{\bfx_1,\dots,\bfx_{p-1}\}$
and fixes $\Phi_n$, so it acts on $X_b$. It is easy to see that 
$X_b/\Sigma_{p-1}$ is birationally isomorphic to the Norm variety 
constructed in \cite[\S2]{SJ} using the hypersurface $W$ defined by $N=b$ 
in the vector bundle of {\it loc.\ cit.} 
By \cite[1.19]{SJ}, $X$ is also a Norm variety.
\end{rem}

\begin{defn}\label{C-move}
A move of type $C_n$ on a sequence $a_1,...,a_n$ in $k^\times$
is a transformation of the kind: 
\[ \text{Type } C_n:\qquad
(a_1,...,a_n)\mapsto (a_1,...,a_{n-2},a_n\Psi_{n-1}(\bfx),a_{n-1}^{-1}).
\]
Here $\Psi_{n-1}$ is a function of $p^{n-1}\!-\!p^{n-2}$ new variables
$\bfx_i=\{\bfx_{1,1},...,\bfx_{1,p-1}\}$.
\end{defn}

By Lemma \ref{lem:basicfact},
$\{ a_1,...,a_n\}=\{ a_1,...,a_{n-2},a_n\Psi_{n-1}(\bfx),a_{n-1}^{-1} \}$,
so the move does not change the symbol in $K^M_n(k)$.
If we do this move $p$ times, always with a new set of variables $\bfx_{i}$,
we obtain a move 
$(a_1,...,a_n)\mapsto (a_1,...,a_{n-2},\gamma_{n-1},\gamma'_{n-1})$
in which $\gamma_{n-1},\gamma'_{n-1}$ are functions of
$p^n-p^{n-1}$ variables $\bfx_{i,j}$, $1\le i<p$, $1\le j\le p$. 

Since these moves do not change the symbol, we have
\begin{equation}\label{eq:moves}
\{ a_1,...,a_n\}=\{ a_1,...,a_{n-2},\gamma_{n-1},\gamma'_{n-1}\}
\end{equation}\goodbreak
in $K^M_n(k)$. The functions $\gamma_{n-1}$
and $\gamma'_{n-1}$ in \eqref{eq:moves} are the ones appearing in 
the Chain Lemma \ref{thm:chainlemma}.

Formally, if $k(\bfx_1)$ is the function field of the move of type $C_n$,
then the function field $F'_n$ of the move \eqref{eq:moves} is the 
tensor product $k(\bfx_1)\oo\cdots\oo k(\bfx_{p})$. We will define a variety
$S_{n-1}$ with function field $F'_n$.

\medskip
Using $p^{n-1}-p^{n-2}$ more variables 
$\bfx'_{i,j}$ ($1\le i<p$, $1\le j\le p$)
we do $p$ moves of type $C_{n-1}$ on $(a_1,...,a_{n-2},\gamma_{n-1})$
to get the sequence 
$(a_1,...,a_{n-3},\gamma_{n-2},\gamma'_{n-2},\gamma'_{n-1})$.
The function field of this move is $F'_{n-1}\oo F'_{n}$, and we will define
a variety $S_{n-2}$ with this function field, together with a morphism
$S_{n-2}\to S_{n-1}$.

Next, apply $p$ moves of type $C_{n-2}$, then $p$ moves of type $C_{n-3}$,
and so on, ending with $p$ moves of type $C_2$. We have the sequence
$(\gamma_1,\gamma'_1,\gamma'_2,...,\gamma'_{n-1})$
in $p^n-p$ variables $\bfx_1,...,\bfx_{p-1}$. Moreover, we see from
Lemma \ref{lem:basicfact} that 
\begin{equation}\label{eq:pmoves}
\{ a_1,\dots,a_n\} =\{\gamma_1,\gamma'_1,\gamma'_2,...,\gamma'_{n-1}\}
\quad\text{ in } K^M_n(k).
\end{equation}
The net effect will be to construct a tower
\begin{equation}\label{eq:Rtower}
S=S_1 \map{f_r} S_{2} \map{} \dotsm\to S_{n-2} \map{} S_{n-1} 
\map{} S_n = \Spec(k).
\end{equation}

Let $S$ be any variety containing $U=\A^{p^n-p}$ as an affine open,
so that $k(S)=k(\bfx_1,...,\bfx_{p-1})$, each $\bfx_i$ is $p^{n-1}$ variables
$x_{i,j}$ and all line bundles on $U$ are trivial.
Then parts (1) and (2) of the Chain Lemma \ref{thm:chainlemma} 
are immediate from \eqref{eq:moves} and \eqref{eq:pmoves}.

Now the only thing to do is to construct $S=S_1$, extend the line bundles (and 
forms) from $U$ to $S$, and prove parts (4) and (6) of \ref{thm:chainlemma}.

\medskip
\section{Model $P_{n-1}$ for moves of type $C_n$}\label{sec:model}

In this section, we construct a tower of varieties $P_r$ and $Q_r$ over $S'$, 
with $p$-forms on lines bundles over them, which will produce a model 
of the forms $\Psi_r$ and $\Phi_r$ in \eqref{eq:Psi} and \eqref{eq:Phi}.
This tower, depicted in \eqref{eq:PQtower}, is defined in 
\ref{PQtower} below.
\addtocounter{equation}{-1}
\begin{equation}\label{eq:PQtower}
P_{n-1}\to\dotsm\to 
P_r \map{} Q_{r-1}\to P_{r-1} \map{}\dotsm\to Q_1\to P_1 \map{}Q_0=S'
\end{equation}

The passage from $S'$ to the variety $P_{n-1}$ is a model for the moves of
type $C_n$ defined in \ref{C-move}.

\medskip

\begin{defn}\label{def:Q-bundle}
Let $X$ be a variety over some fixed base $S'$.
Given line bundles $K$, $L$ on $X$, we can form the vector bundle 
$V=\cO\oplus L$, the $\bP^1$-bundle $\bP(V)$ over $X$,
and $\bL$. Taking products over $S'$, set
\[ 
P = \prod\nolimits_1^{p-1}\bP(\cO\oplus L); 
\quad Q=X\times_{S'} P 
\]
On $P$ and $Q$,
we have the exterior products of the tautological line bundles:
\[
\bLbox=\bL\boxtimes\bL\boxtimes\cdots\boxtimes\bL
\textrm{ on } P ,\quad K\boxtimes\bLbox\textrm{ on } Q.
\]
Given $p$-forms $\varphi$ and $\phi$ on $K$ and $L$, respectively,
the line bundle $\bL$ has the $p$-form $\epsilon$, as in 
Example \ref{P(O+K)}, and the line bundles
 $\bLbox$ and $K\boxtimes\bLbox$ are equipped with the 
product $p$-forms $\Psi=\prod\epsilon$ and $\Phi=\varphi\oo\Psi$.
\end{defn}

\begin{subrem}
Let $x=(x_1,\dots,x_{p-1})$ denote the generic point of $X^{p-1}$. The
function fields of $P$ and $Q$ are $k(P)=k(x)(y_1,\dots,y_{p-1})$ and 
$k(Q)=k(x_0)\otimes k(P)$.  We may represent their generic points
in coordinate form as a 
$(p-1)$-tuple $\{(1:y_i)\}$, where the $y_i$ generate $L$ over $x_i$.  
Then $y=\{(1,y_i)\}$ is a generator of $\bLbox$ at the generic
point, and $\Psi(y)=\prod(1-\phi(y_i))$, $\Phi(y)=\varphi(x_0)\Psi(y)$.
\end{subrem}

\begin{subex}\label{ex:Q-bundle}
An important special case arises when we begin with two line bundles
$H$ on $S'$, $K$ on $X$, with $p$-forms $\alpha$ and $\varphi$. 
In this case, we set $L=H\otimes K$ and equip it with the product form 
$\phi(u\otimes v)=\alpha(u)\varphi(v)$. At the generic point $q$ of $Q$
we can pick a generator $u\in H|_q$ and set $y_i=u\otimes v_i$; the
forms resemble the forms of \eqref{eq:Psi} and \eqref{eq:Phi}:
\[ \Psi(y)=\prod\bigl(1-\alpha(u)\varphi(v_i)\bigr), \quad
\Phi(y)=\varphi(v_0)\,\Psi(y). \]
\end{subex}

\begin{subrem}
Suppose a group $G$ acts on $S'$, $X$, $K$ and $L$, and $K_0$, $L_0$ are
nontrivial 1-dimensional representations so that 
at every fixed point $x$ of $X$ (a) $k(x)=k$, (b) 
$L_x\cong L_0$. Then $G$ acts on $P$ (resp., $Q$) with $2^{p-1}$ fixed 
points $y$ over each fixed point of $X^{p-1}$ (resp., of $X^p$), 
each with $k(y)=k$, and each fiber of $\bL=\bLbox$ 
(resp., $K\boxtimes\bL$) is the representation 
$L_0^{j}$ (resp., $K_0\otimes L_0^{j}$) for some $j$ ($0\le j<p$).
Indeed, $G$ acts nontrivially on each term $\bP^1$ of the fiber 
$\prod \bP^1$, so that the fixed points in the fiber are the points
$(y_1,...,y_{p-1})$ with each $y_i$ either $(0:1)$ or $(1:0)$.
\end{subrem}

\medskip

We now define the tower \eqref{eq:PQtower} of $P_r$ and $Q_r$ over a 
fixed base $S'$, by induction on $r$.  We start with line bundles 
$H_1,\dots,H_r$, and $K_0=\cO_{S'}$ on $S'$, and set $Q_0=S'$.

\begin{defn}\label{PQtower}
Given a variety $Q_{r-1}$ and a line bundle $K_{r-1}$ on $Q_{r-1}$, 
we form the varieties $P_{r}=P$ and $Q_{r}=Q$ using the construction in
Definition \ref{def:Q-bundle}, with $X\!=\!Q_{r-1}$, $K\!=\!K_{r-1}$ and 
$L=H_{r}\oo K_{r-1}$ as in \ref{ex:Q-bundle}. 
To emphasize that $P_r$ only depends upon $S'$ and $H_1,\dots,H_r$, we will
sometimes write $P_r(S';H_1,\dots,H_r)$.
As in \ref{def:Q-bundle}, $P_r$ has the line bundle $\bLbox$, 
and $Q_r$ has the line bundle $K_r=K_{r-1}\boxtimes\bLbox$.

Suppose that we are given $p$-forms $\alpha_i\ne0$ on $H_i$, and we set
$\Phi_0(t)=t^p$ on $K_0$. Inductively, the line bundle $K_{r-1}$ on $Q_{r-1}$ 
is equipped with a $p$-form $\Phi_{r-1}$. As described in 
\ref{def:Q-bundle} and \ref{ex:Q-bundle}, the line bundle
$\bLbox$ on $P_r$ obtains a $p$-form $\Psi_r$ from the $p$-form 
$\alpha_r\oo\Phi_{r-1}$ on $L=H_r\oo K_{r-1}$, and $K_r$ obtains a $p$-form
$\Phi_r=\Phi_{r-1}\oo\Psi_r$.
\end{defn}
\vspace{-5pt}

\begin{subex}
$Q_1=P_1$ is $\prod_1^{p-1}\bP^1(\cO\oplus H_1)$ over $S'$, 
equipped with the line bundle $K_1=\bLbox$. If $H_1$ is a
trivial bundle with $p$-form $\alpha_1(t)=a_1t^p$ 
then $\Phi_1$ is the $p$-form $\Phi_1$ of Example \ref{ex:r=1}.  

$P_2$ is $\prod_1^{p-1}\bP^1(\cO\oplus H_2\oo K_1)$ over
$Q_1^{p-1}$, and $K_2=K_1\boxtimes\bLbox$. 
\end{subex}

\begin{lem}\label{dim/n=2}
If  $r>0$ then $\dim(P_r/S')=(p^r-p^{r-1})$
and $\dim(Q_r/S')=p^r-1$. 
\end{lem}

\begin{proof} 
Set $d_r=\dim(Q_r/S')$.
This follows easily by induction from the formulas
$\dim(P_{r+1}/S')=(p-1)(d_r+1)$, 
$\dim(Q_{r+1}/S')=p(d_r+1)-1$.
\end{proof}

Choosing generators $u_i$ for $H_i$ at the generic point of $S'$,
we get units $a_i=\alpha_i(u_i)$.

\begin{lem}\label{lem:agreegeneric}
At the generic points of $P_r$ and $Q_r$, the $p$-forms 
$\Psi_n$ and $\Phi_n$ of \ref{PQtower}
agree with the forms defined in \eqref{eq:Psi} and \eqref{eq:Phi}.
\end{lem}

\begin{proof}
This follows by induction on $r$, using the analysis of \ref{ex:Q-bundle}.
Given a point $q=(q_1,\dots,q_p)$ of $Q_{r-1}^{p-1}$ and a point
$\{(1:y_i)\}$ on $P_r$ over it, $y=\{(1,y_i)\}$ is a nonzero point 
on $\bLbox$ and $y_i=1\otimes v_i$ for a section $v_i$ of $K_{r-1}$.
Since $\epsilon(1,y_i)=1-a_r\Phi_{r-1}(v_i)$ and 
$\Psi_r(y)=\prod\epsilon(1,y_i)$, the forms $\Psi_r$ agree.
Similarly, if $v_0$ is the generator of $K_{r-1}$ over the generic
point ${q_0}$ then $y'=v_0\otimes y$ is a generator of $K_r$ and
\[  \Phi_r(y') = \Phi_{r-1}(v_0) \Psi_r(y),
\]
which is also in agreement with the formula in \eqref{eq:Phi}.
\end{proof}

Recall that $K_0$ is the trivial line bundle, and that $\Phi_0$ is the
standard $p$-form $\Phi_0(v)=v^p$ on $K_0$. Every point of
$P_r\!=\prod\bP(\cO\oplus L)$ has the form
$w=(w_1,\dots,w_{p-1})$, and the projection 
$P_r\to\prod Q_{r-1}$ sends $w\in P_r$ 
to a point $x=(x_1,\dots,x_{p-1})$. 

\begin{prop}\label{part4}
Let $s\in S'$ be a point such that $a_1|_s,\dots,a_r|_s\ne0$. \\
1. If $\Psi_r|_w=0$ for some $w\in P_r$, then
$\{ a_1,\dots,a_r\}$ vanishes in $K^M_r(k(w))/p$. \\
2. If $\Phi_r|_q=0$ for some $q=(x_0,w)\in Q_r$,
$\{ a_1,\dots,a_r\}$ vanishes in $K^M_r(k(q))/p$.
\end{prop}

\begin{proof}
Since $\Phi_r=\Phi_{r-1}\otimes\Psi_r$, the assumption that $\Psi_r|_w=0$
implies that $\Phi_r|_q=0$ for any $x_0\in Q_{r-1}$ over $s$. 
Conversely, if $\Phi_r|_q=0$ then either $\Psi_r|_w=0$ or 
$\Phi_{r-1}|_{x_0}=0$.
Since $\Phi_0\ne0$, we may proceed by induction on $r$ and assume that
$\Phi_{r-1}|_{x_j}\ne0$ for each $j$,
so that $\Phi_r|_q=0$ is equivalent to $\Psi_r|_w=0$.

By construction, the $p$-form on $L=H_r\otimes K_{r-1}$ is 
$\phi(u_r\otimes v)=a_r\Phi_{r-1}(v)$, where $u_r$ generates the vector 
space $H_r|_s$ and $v$ is a section of $K_{r-1}$. Since $\Psi_r|_w$ is the
product of the forms $\epsilon|_{w_j}$, some $\epsilon|_{w_j}=0$.
Lemma \ref{lem:pth-power} implies that 
$a_r\Phi_{r-1}(v)$ is a $p$th power in $k({x_j})$, and hence in $k(w)$,
for any generator $v$ of $K_{r-1}|_{x_j}$.
By Lemma \ref{lem:basicfact},
$\{ a_1,\dots,a_{r-1},\Phi_{r-1}\}=0$ and hence
\[ 
\{ a_1,\dots,a_r\} = \{ a_1,\dots,a_{r-1}, a_r\Phi_{r-1} \} = 0
\]
in $K^M_r(k(w))/p$, as claimed.
\end{proof}

We conclude this section with some identities in $CH(P_n)/p\,CH(P_n)$,
given in \ref{cor:RtoCH}.
To simplify the statements and proofs below,
we write $\CH(X)$ for $CH(X)/p\,CH(X)$, and adopt the following notation.

\begin{defn}
Set $\eta=c_1(H_n)\in \CH^1(S')$, and $\gamma=c_1(\bLbox)\in\CH^1(P_n)$.
Writing $\bP$ for the bundle $\bP(\cO\oplus H_n\oo K_{n-1})$ over $Q_{n-1}$, 
let $c \in \CH(\bP)$ denote $c_1(\bL)$ and let $\kappa\in\CH(Q_{n-1})$ 
denote $c_1(K_{n-1})$. We write $c_j,\kappa_j\in\CH(P_n)$ for
the images of $c$ and $\kappa$ under the $j$th coordinate pullbacks 
$\CH(Q_{n-1})\to\CH(\bP)\to\CH(P_n)$.
\end{defn}

\begin{lem}\label{Lemma12}
Suppose that $H_1,\dots,H_{n-1}$ are trivial. Then

(a) $\gamma^{p^n}=\gamma^{p^{n-1}}\eta^d$ in $\CH(P_n)$, 
where $d=p^n-p^{n-1}$;

(b) If in addition $H_n$ is trivial, then $\gamma^d=-\prod c_j \kappa_j^e$,
where $e=p^{n-1}-1$.

(c) If $S'=\Spec k$ then the zero-cycles $\kappa^e\in\CH_0(Q_{n-1})$
and $\gamma^d\in\CH_0(P_n)$ have 
\[ \deg(\kappa^e)\equiv(-1)^{n-1} \quad\text{and}\quad 
\deg(\gamma^d)\equiv-1 \quad\text{modulo~} p. \]
\end{lem}
\goodbreak

\begin{proof}
First note that because $K_{n-1}$ is defined over the
$e$-dimensional variety $Q_{n-1}(\Spec k;H_1,...,H_{n-1})$,
the element $\kappa=c_1(K_{n-1})$ 
satisfies $\kappa^{p^{n-1}}\!=0$. Thus 
$(\eta+\kappa)^{p^{n-1}}=\eta^{p^{n-1}}$ and hence $(\eta+\kappa)^d=\eta^d$.
Now the element $c=c_1(\bL)$ satisfies the relation 
$c^2=c(\eta+\kappa)$ in $\CH(\bP)$ and hence
\[
c^{p^n}=c^{p^{n-1}}(\eta+\kappa)^d = c^{p^{n-1}}\eta^d
\]
in $\CH^{p^n}(\bP)$. Now recall that $P_n=\prod\bP$.
Then $\gamma=\sum c_j$ and
\[
\gamma^{p^n} = \sum c_j^{p^n}=\sum c_j^{p^{n-1}}\eta^d
	= \gamma^{p^{n-1}}\eta^d. 
\]
When $H_n$ is trivial we have $\eta=0$ and hence $c^2=c\,\kappa$.
Setting $b_j=c_j^{p^{n-1}}=c_j\kappa_j^e$, we have 
$\gamma^d=\gamma^{p^{n-1}(p-1)}=(\sum b_j)^{p-1}$.
To evaluate this, we use the algebra trick that since $b_j^2=0$ for
all $j$ and $p=0$ we have $(\sum b_j)^{p-1}=(p-1)!\prod b_j=-\prod b_j$.

For (c), note that if $S'=\Spec k$ then $\eta=0$ and $\gamma^d$ is a 
zero-cycle on $P_n$. By the projection formula for $\pi:P_n\to\prod Q_{n-1}$,
part (b) yields $\pi_*\gamma^d=(-1)^p\prod \kappa_j^e$. 
Since each $Q_{n-1}$ is an iterated projective space bundle,
$CH(\prod Q_{n-1})=\otimes_1^{p-1}CH(Q_{n-1}$, and the degree of
$\prod \kappa_j^e$ is the product of the degrees of the $\kappa_j^e$.
By induction on $n$, these degrees are all the same, and nonzero, so 
$\deg(\prod\kappa_j^e)\equiv1\pmod{p}$.

It remains to establish the inductive formula for $\deg(\kappa^e)$.
Since it is clear for $n=0$, and the $Q_i$ are projective space bundles,
it suffices to compute that $c_1(K_n)^{p^n-1}=\kappa^e\gamma^d$ in
$\CH(Q_n)=\CH(Q_{n-1})\otimes\CH(P_n)$. 
Since $\kappa^{e+1}=0$ and $c_1(K_n)=\kappa+\gamma$ we have
\[
c_1(K_n)^{p^{n-1}}=\kappa^{e+1}+\gamma^{p^{n-1}}=\gamma^{p^{n-1}},
\]
and hence $c_1(K_n)^d=\gamma^d$. Since $\gamma^{d+1}=0$, 
this yields the desired calculation:
\[
c_1(K_n)^{p^n-1} = c_1(K_n)^e c_1(K_n)^d = (\kappa+\gamma)^e\gamma^d
= \kappa^e\gamma^d.  \qedhere
\]
\end{proof}

\begin{cor}\label{cor:RtoCH}
There is a ring homomorphism 
$\bF_p[\lambda,z]/(z^p-\lambda^{p-1}z)\to \CH(P_n)$,
sending $\lambda$ to $\eta^{p^{n-1}}$ and $z$ to $\gamma^{p^{n-1}}$.
\end{cor}
\goodbreak

\section{Model for $p$ moves}\label{sec:ModelforMoves}

In this section we construct maps $S_{n-1}\to S_n$ which
model the $p$ moves of type $C_n$ defined in \ref{C-move}.
Each such move introduces $p^{n-1}-p^{n-2}$ new variables, and
will be modelled by a map $Y_r\to Y_{r-1}$ of relative dimension 
$p^{n-1}-p^{n-2}$, using the $P_{n-1}$ construction in \ref{PQtower}. 
The result (Definition \ref{Ytower}) will be a tower of the form:
\begin{equation*}
\begin{matrix}
J_{n-1}=L_p & & L_{p-1} & & L_2 && L_1~\phantom{\lra}&\hspace{-5pt} 
L_0 = J_n \quad \\
S_{n-1}= Y_p&\map{f_p}& Y_{p-1} &\map{} \dotsm\to& Y_{2} &\map{f_2}& 
Y_{1} \map{f_1}&\hspace{-9pt} Y_0 = S_n.
\end{matrix}
\end{equation*}

Fix $n\ge2$, a variety $S_n$, and line bundles 
$H_1,\dots,H_{n-2}$, $H_{n}$ and $J_n$ on $S_n$. The first step in 
the tower is to form $Y_0=S_n$ and 
$Y_1 = P_{n-1}(S_n;H_1,\dots,H_{n-2},J_n)$, with line bundles
$L_0=J_n$ and $L_1=H_n\oo\bLbox$ as in \ref{PQtower}. 
In forming the other $Y_r$, the base in the $P_{n-1}$ construction 
\ref{PQtower} will become $Y_{r-1}$ and only the final line bundle
will change (from $J_n$ to $L_{r-1}$). Here is the formal definition.

\begin{defn}\label{Ytower}
For $r>1$, we define morphisms $f_{r}:Y_{r}\to Y_{r-1}$ and line bundles 
$\bL_r^\boxtimes$ and $L_r$ on $Y_r$ as follows. Inductively, we are
given a morphism $f_{r-1}:Y_{r-1}\to Y_{r-2}$ and line
bundles $L_{r-1}$ on $Y_{r-1}$, $L_{r-2}$ on $Y_{r-2}$.
Set $\bL_r^\boxtimes=\bLbox$,
\begin{equation*}
Y_r = P_{n-1}(Y_{r-1};H_1,\dots,H_{n-2},L_{r-1})\map{f_r}Y_{r-1}, 
\quad L_r=f_r^*f_{r-1}^*(L_{r-2})\oo\bL_r^\boxtimes.  
\end{equation*}
\goodbreak
\noindent
Finally, we write $S_{n-1}$ for $Y_p$ and set $J_{n-1}=L_p$,
$J'_{n-1}=f_{p}^*(L_{p-1})$.  By Lemma \ref{dim/n=2}, 
$\dim(Y_r/Y_{r-1})=p^{n-1}-p^{n-2}$ and hence
$\dim(S_{n-1}/S_n) = p^n-p^{n-1}$.
\end{defn}

For example, when $n=2$ and and $H_1$ is trivial, this tower is exactly the tower of \ref{def:tower}: we have 
$Y_r=P_1(Y_{r-1};L_{r-1})=\prod\bP^1(\cO\oplus L_{r-1})$. 

\begin{subrem}
The line bundles $J_{n-1}$ and $J_{n-1}'$ will be the line bundles of the
Chain Lemma \ref{thm:chainlemma}. The rest of tower \eqref{eq:Rtower}
will be obtained in Definition \ref{Rtower} by repeating this construction 
and setting $S=S_1$.
\end{subrem}

The rest of this section, culminating in Theorem \ref{thm:part6},
is devoted to proving part (6) of the Chain Lemma, that the
degree of the zero-cycle $c_1(J_1)^{\dim\, S}$ is relatively prime to $p$.
In preparation, we need to compare
the degrees of the zero-cycles $c_1(J_{n-1})^{\dim S_{n-1}}$ on $S_{n-1}$
and $c_1(J_{n})^{\dim S_{n}}$ on $S_{n}$. 
In order to do so, we introduce the following algebra.

\begin{defn}\label{def:Ar} 
We define the graded $\bF_p$-algebra $A_r$ and $\bar A_r$ by
$\bar A_r=A_r/\lambda_{-1}A$ and:
\[ A_r = 
\bF_p[\lambda_{-1},\lambda_0,\dots,\lambda_r,z_1,\dots,z_r]/
(\{z_i^p-\lambda_{i-1}^{p-1}z_i, \lambda_i-\lambda_{i-2}-z_i
~\vert~ i= 1,\dotsc r\}).
\]
\end{defn}

\begin{subrem}\label{RtoCHY}
By Corollary \ref{cor:RtoCH},  
there is a homomorphism $A_p\smap{\rho}\CH(Y_p)$, sending $\lambda_r$ to 
$c_1(L_r)^{p^{n-2}}$ and $z_r$ to $c_1(\bL_r^\boxtimes)^{p^{n-2}}$\!. 
When $H_{n-1}$ is trivial, $\rho$ factors through $\bar A_p$.
\end{subrem}

\begin{lem}\label{lem:19}
In $\bar A_r$, every element $u$ of degree~1 satisfies 
$u^{p^2}=u^p\lambda_0^{p^2-p}$.
\end{lem}

\begin{proof}
We will show that $\bar A_r$ embeds into a product of graded rings of the form
$\Lambda_k=\bF_p[\lambda_0][v_1,\dots,v_k]/(v_1^p,\dots,v_k^p)$.
In each entry, $u=a\lambda_0+v$ with $v^p=0$ and $a\in\bF_p$, so
$u^p=a\lambda_0^p$ and $u^{p^2}=a\lambda_0^{p^2}$, whence the result.

Since $\bar A_{r+1}=\bar A_{r}[z]/(z^p-\lambda_{r}^{p-1}z)$ is flat over
$\bar A_r$, it embeds by induction into a product of graded rings of the form
$\Lambda'=\Lambda_k[z]/(z^p-u^{p-1}z)$, $u\in\Lambda_k$. 
If $u\ne0$, there is an embedding of $\Lambda'$ into 
$\prod_{i=0}^{p-1}\Lambda_k$ whose $i$th component sends $z$ to $iu$.
If $u=0$, then $\Lambda'\cong\Lambda_{k+1}$.
\end{proof}

\begin{subrem}\label{rem:indep}
It follows that if $m>0$ and $(p^2-p)\mid m$ then
$u^{kp+m}=\lambda_0^m u^{kp}$.
\end{subrem}

\begin{prop}\label{prop:cor22} In $\bar A_r$, 
 $\lambda_p^{p^N-p}=\lambda_0^{p^N-p^2}(\prod z_i^{p-1}\!+T\lambda_0)$,
where $\deg(T)\!=\!p^2\!-\!p\!-\!1$.
\end{prop}

\begin{proof}
By Definition \ref{def:Ar}, $\bar A_p$ is free over $\bF_p[\lambda_0]$, 
with the elements $\prod z_i^{m_i}$ ($0\le m_i<p$) forming a basis. 
Thus any term of degree $p^N-p$ is a linear combination of 
$F=\lambda_0^{p^N-p^2}\prod z_i^{p-1}$ and terms of
the form $\lambda_0^{m_0}\prod z_i^{m_i}$ where $\sum m_i=p^N-p^2$
and $m_0>p^N-p^2$. It suffices to determine the coefficient of $F$ 
in $\lambda_p^{p^N-p}$. Since 
$\lambda_p^{p^N-p}=\lambda_0^{p^N-p^2}\lambda_p^{p^2-p}$ by
Remark \ref{rem:indep}, it suffices to consider $N=2$, when 
$F=\prod z_i^{p-1}$.

As in the proof of Proposition \ref{part6/n=2}, if $p\ge r\ge2$ we 
compute in the ring $\bar A_r$ that
\begin{align*}
\lambda_r^{r(p-1)} =& (z_{r}+\lambda_{r-2})^{p(r-1) + (p-r)} =
(z_{r}^p + \lambda_{r-2}^p)^{r-1}\!\cdot (z_r+\lambda_{r-2})^{p-r} \\
=& (z_r \lambda_{r-1}^{p-1} + 
		\lambda_{r-2}^p)^{r-1} (z_r+\lambda_{r-2})^{p-r}
= z_r^{p-1}\lambda_{r-1}^{(r-1)(p-1)} + T,
\end{align*}
where $T\!\in \bar A_{r-1}[z_r]$ is a homogeneous polynomial of 
total degree $<\!p\!-\!1$ in $z_r$. By induction on $r$,
the coefficient of $(z_1\cdots z_r)^{p-1}$ in $\lambda_r^{r(p-1)}$
is~1 for all $r$.
\end{proof}

\begin{lem}\label{thm16base}
If $S_n=\Spec(k)$ and $c=c_1(J_{n-1})\in CH^1(S_{n-1})$, then
\[ \deg(c^{\dim S_{n-1}})\equiv1\pmod{p}. \]
\end{lem}

\begin{proof}
Set $d=\dim(S_{n-1})=p^n-p^{n-1}$; under the map 
$A_p\smap{\rho}\CH(S_{n-1})$
of \ref{RtoCHY}, the degree $p^2-p$ part of $A_p$ maps to $CH^d(S_{n-1})$.
In particular, the zero-cycle $c^{d}=\rho(\lambda_p)^{p^2-p}$ 
equals the product of the $\rho(z_i)^{p-1}=c_1(\bL_i^\boxtimes)^{d/p}$ 
by Proposition \ref{prop:cor22} (the $T\lambda_0^\ast$ term maps to
zero for dimensional reasons). Because $S_{n-1}=Y_p$ is a product of 
iterated projective space bundles, $CH_0(Y_p)$ is the tensor product of their
$CH_0$ groups, and the degree of $c^d$ is the product of the
degrees of the $c_1(\bL_i^\boxtimes)^{d/p}$, 
each of which is $-1$ by Lemma \ref{Lemma12}. 
It follows that $\deg(c^d)\equiv1\pmod{p}$.
\end{proof}

\begin{thm}\label{Theorem16}
If $S_n$ has dimension $p^M-p^n$ and $H_1,\dots,H_{n-1}$ are trivial
then the zero-cycles $c_1(J_{n-1})^{\dim S_{n-1}} \in CH_0(S_{n-1})$
and $c_1(J_{n})^{\dim S_{n}} \in CH_0(S_{n})$ have the same 
degree modulo $p$:
$$\deg(c_1(J_{n-1})^{\dim S_{n-1}})=\deg(c_1(J_n)^{\dim S_n})\pmod{p}.$$
\end{thm}

\begin{proof}
By \ref{RtoCHY}, 
there is a homomorphism $A_p\smap{\rho} \CH(S_{n-1})$, sending $\lambda_r$ to 
$c_1(L_r)^{p^{n-2}}$ and $z_r$ to $c_1(\bL_r^\boxtimes)^{p^{n-2}}$\!. 
Because $H_{n-1}$ is trivial, $\rho$ factors through $\bar A_p$.

Set $N=M-n+2$ and $y=\lambda_0^{p^{N}-p^2}$\!, so 
$\rho(y)=c_1(J_n)^{\dim S_n}\in\CH_0(S_n)$. 
From Proposition \ref{prop:cor22} we have
$\lambda_p^{p^{N}-p}\equiv y\prod z_i^{p-1}$ modulo $\ker(\rho)$.
From Lemma \ref{lem:deg(yz)}, the degree of this element equals the degree of
$y$ modulo~$p$.
\end{proof}

\subsection*{The $p$-forms}
We now define the $p$-forms on the line bundles $J_{n-1}$ and $J'_{n-1}$.
using the tower \eqref{Ytower}. 
Suppose that the line bundles $L_{-1}=H_n$ and $L_0=J_n$ on $S_n$ are 
equipped with the $p$-forms $\beta_{-1}$ and $\beta_0$. 
We endow the line bundle $L_1$ in Definition \ref{Ytower} with the 
$p$-form $\beta_1=f^*(\beta_{-1})\oo\Psi_{n-1}(\beta_0)$;
inductively, we endow the line bundle $L_r$ with the $p$-form 
\[
\beta_r= f^*(\beta_{r-2})\oo\Psi_{n-1}(\beta_{r-1}).
\]

\begin{exx} 
When $n=2$ and $H_1$ is trivial saw that
the tower \ref{Ytower} is exactly the tower of \ref{def:tower}. In addition,
the $p$-form $\beta_r=\Psi_1(\beta_{r-1})$ agrees with the $p$-form
$\varphi_r=f^*(\varphi_{r-2})\oo\Psi_r$ of \ref{def:gamma2}.
\end{exx} 

\begin{lem}\label{lem:agreePsi}
If $\beta_0=\alpha_{n-1}$ and $\beta_{-1}=\alpha_n$, then
(at the generic point of $Y_1$)
the $p$-form $\beta_p$ agrees with the form
$\alpha_n\Psi_{n-1}$ in \eqref{C-move}.
\end{lem}

\begin{proof}
By Lemma \ref{lem:agreegeneric}, the form agrees with the form
of \eqref{eq:Psi}.
\end{proof}

\begin{defn}\label{Rtower}
The tower \eqref{eq:Rtower} of varieties $S_i$ is obtained by
downward induction, starting with $S_n=\Spec(k)$ and $J_n=H_{n-1}$. 
Construction \ref{Ytower}
yields $S_{n-1}$, $J_{n-1}$ and $J'_{n-1}$. Inductively, we repeat 
construction \ref{Ytower} for $i$, starting with the output $S_{i+1}$ and 
$J_{i+1}$ of the previous step, to produce $S_{i}$, $J_{i}$ and $J'_{i}$.

By downward induction in the tower \eqref{eq:Rtower}, each $J_i$ and
$J'_i$ carries a $p$-form, which we call $\gamma_i$ and $\gamma'_i$,
respectively. By \ref{lem:agreePsi}, these forms agree with the forms
$\gamma_i$ and $\gamma'_i$ of \eqref{eq:moves} and \eqref{eq:pmoves}.
\end{defn}

Since $\dim(S_i/S_{i-1})=p^{i+1}-p^{i}$ we have $\dim(S_i/S_n)=p^n-p^i$.
Thus if we combine Lemma \ref{thm16base} and Theorem \ref{Theorem16},
we obtain the following result.

\begin{thm}\label{thm:part6}
For each $i<n$, $\deg(c_1(J_i)^{\dim\,S_i})\equiv-1\pmod{p}.$
\end{thm}

Theorem \ref{thm:part6} establishes part (6) of the 
Chain Lemma \ref{thm:chainlemma}, that $\deg(c_1(J_1)^{\dim\, S_1})$.

\begin{proof}[Proof of the Chain Lemma \ref{thm:chainlemma}]
We verify the conditions for the variety $S=S_1$ in the tower 
\eqref{eq:Rtower}; the line bundles $J_i$ and $J'_i$ and their $p$-forms
are obtained by pulling back from the bundles and forms defined 
in \ref{Rtower}.
Part (1) of Theorem \ref{thm:chainlemma} is immediate from the 
construction of $S$; part (6) is Theorem \ref{thm:part6}, combined with
Lemma \ref{thm16base}. Part (2) was just established, and
part (4) was proven in Proposition \ref{part4};
parts (3) and (5) follow from (2) and (4).
This completes the proof of the Chain Lemma.
\end{proof}

\section{Nice $G$-actions}\label{sec:trucking}

We will extend the Chain Lemma to include an action by $G=\mu_p^n$
on $S$, $J_i$, $J'_i$ leaving $\gamma_i$ and $\gamma'_i$ invariant,
such that the action is admissible in the following sense.

\begin{defn}\label{def:Gnice} (Rost, cf.\ \cite[p.2]{Rost-CSV})
Let $G$ be a group acting on a $k$-variety $X$. We say that the action is 
{\it nice} if $\Fix_G(X)$ is 0-dimensional, and consists of $k$-points.

When $G$ also acts on a line bundle $L$ over $X$, the action on the 
geometric bundle $L$ is {\it nice} exactly when $G$ acts nontrivially 
on $L|_x$ for every fixed point $x\in X$, and in this case $\Fix_G(L)$ 
is the zero-section over $\Fix_G(X)$.

Suppose that $G$ acts nicely on each of several line bundles $L_i$ over $X$.
We say that $G$ {\it acts nicely} on $\{ L_1,\dots,L_r\}$ if for each 
fixed point $x\in X$ the image of the canonical representation
$G\to\prod\Aut(L_i|_x)=\prod k(x)^\times$ is $\prod G_i$, 
with each $G_i$ nontrivial.
\end{defn}

\begin{subrem}\label{rem:Gproduct}
If $X_i\to S$ are equivariant maps and the $X_i$ are nice, then 
$G$ also acts nicely on $X_1\times_SX_2$. However, even if $G$ acts
nicely on line bundles $L_i$ it may not act nicely on
$L_1\boxtimes L_2$, because the representation over $(x_1,x_2)$
is the product representation $L_1|_{x_1}\otimes L_2|_{x_2}$.
\end{subrem}

\begin{ex}\label{ex:Gprojective}
Suppose that $G$ acts nicely on a line bundle $L$ over $X$. Then
the induced $G$-action on $\bP=\bP(\cO\oplus L)$ and its 
canonical line bundle $\bL$ is nice. Indeed, if $x\in X$ is a fixed point
then the fixed points of $\bP|_x$ consist of the two $k$-points 
$\{[\cO], [L]\}$, and if $L|_x$ is the representation $\rho$ then
$G$ acts on $\bL$ at these fixed points as $\rho$ and $\rho^{-1}$, 
respectively.

By \ref{rem:Gproduct}, $G$ also acts nicely on the products 
$P=\prod\bP(\cO\oplus L)$ and $Q=X\times_{S'}P$ of 
Definition \ref{def:Q-bundle}, but it does not act nicely on $\bLbox$.
\end{ex}

\begin{ex}\label{ex:GKummer}
The group $G$ also acts nicely on the Kummer algebra bundle $\cA=\cA(L)$ 
of \ref{Kummeralgebra}, and on its projective space $\bP(\cA)$. 
Indeed, an elementary calculation shows that $\Fix_G\bP(\cA)$
consists of the $p$ sections $[L^i]$, $0\le i<p$ over $\Fix_G(X)$. 
In each fiber, the (vertical) tangent space at each fixed point is the 
representation $\rho\oplus\cdots\oplus \rho^{p-1}$. If $G=\mu_p$,
this is the reduced regular representation.

Over any fixed point $x\in X$, $L|_x$ is trivial, and the symmetric 
group $\Sigma_p$ acts on the bundle $\cA|_x$, permuting the fixed points.
This induces isomorphisms between the tangent spaces at these points.
\end{ex}

\begin{subex}\label{ex:GB}
The action of $G$ on $Y=\bP(\cO\oplus\cA)$ is not nice. In this case, an
elementary calculation shows that $\Fix_G(Y)$ consists of the points
$[L^i]$ of $\bP(\cA)$, $0<i<p$, together with the projective line 
$\bP(\cO\oplus\cO)$ over every fixed point $x$ of $X$. 
For each $x$, the (vertical) tangent space at $[L^i]$ is 
$1\oplus\rho\oplus\cdots\oplus \rho^{p-1}$; 
if $G=\mu_p$, this is the regular representation.
\end{subex}

When $G=\mu_p^n$, the following lemma allows us to assume that
the action on $L|_x$ is induced by the standard representation 
$\mu_p\subset k^\times$, via a projection $G\to\mu_p$. 
\hspace{-3pt}

\begin{lem}\label{lem:Glines}
Any nontrivial 1-dimensional representation $\rho$ of $G=\mu_p^n$
factors as the composition of a projection $G\to\mu_p$ with the
standard representation of $\mu_p$.
\end{lem}

\begin{proof}
The representation $\rho$ is a nonzero element of 
$(\Z/p)^n=G^*=\Hom(\mu_p^n,\mathbb{G}_m)$, and $\pi$ is the 
Pontryagin dual of the induced map $\Z/p\to G^*$ 
sending 1 to $\rho$.
\end{proof}

The construction of the $P_r$ and $Q_r$ in \ref{PQtower} is natural in
the given line bundles $H_1,\dots,H_n$ over $S'$, and so is the 
construction of the $Y_r$, $S_r$ and $S$ in \ref{Ytower} and \ref{Rtower}.
Since $\prod_{i=1}^n\Aut(H_i)$ acts on the $H_i$, this group 
(and any subgroup) will act on the variety $S$ of the Chain Lemma.  
We will show that it acts nicely on $S$.

Recall from Definition \ref{PQtower} that $P_r$ and $Q_r$ are defined by
the construction \ref{def:Q-bundle} using the line bundle
$L_r=H_r\oo K_{r-1}$ over $Q_{r-1}$.

\begin{lem}\label{lem:nicePQ}
If $S'=\Spec(k)$, then $G=\mu_p^r$ acts nicely on $L_r$, $P_r$ and $Q_r$.
\end{lem}

This implies that any subgroup of $\prod_{i=1}^r\Aut(H_i)$ 
containing $\mu_p^r$ also acts nicely.

\begin{proof}
We proceed by induction on $r$, the case $r=1$ being \ref{ex:Gprojective},
so we may assume that $\mu_p^{r-1}$ acts nicely on $Q_{r-1}$.
By \ref{rem:Gproduct}, it suffices to show that $G=\mu_p^r$ acts nicely on
$\bP(\cO\oplus L_r)$, where $L_r=H_r\oo K_{r-1}$. 
Since the final component $\mu_p$ of $G$ acts trivially on $K_{r-1}$
and $Q_{r-1}$ and nontrivially on $H_r$, $G=\mu_p^{r-1}\times\mu_p$ 
acts nicely on $L_r$. By Example \ref{ex:Gprojective}, $G$ acts nicely
on $\bP(\cO\oplus L_r)$. 
\end{proof}
\goodbreak

The proof of Lemma \ref{lem:nicePQ} goes through in slightly greater
generality.

\begin{cor}\label{cor:nicePQ}
Suppose that $G=\mu_p^n$ acts nicely on $S'$ and on the
line bundles $\{ H_1,\dots,H_r\}$ over it.
Then $G$ acts nicely on $L_r$, $P_r$ and $Q_r$.
\end{cor}

\begin{proof}
Without loss of generality, we may replace $S'$ by a fixed point $s\in S'$,
in which case $G$ acts nicely on $\{ H_1,\dots,H_r\}$ through the surjection 
$\mu_p^n\to\mu_p^{r}$. Now we are in the situation of Lemma \ref{lem:nicePQ}.
\end{proof}

\begin{subex}\label{ex:niceY}
Since $\mu_p^{n-1}$ acts nicely on $Y=P_{n-1}(S';H_1,\dots,H_{n-1})$ and 
on the bundle $K_{n-1}$, while $\mu_p$ of $G=\mu_p^n$ acts solely on $H_n$,
it follows that the group $\mu_p^n=\mu_p^{n-1}\times\mu_p$ acts nicely on
$\{ H_1,\dots,H_{n-1},H_n\oo\bLbox \}$ over $Y$.
\end{subex}

We can now process the tower of varieties $Y_r$ defined in \ref{Ytower}.
For notational convenience, we write $H_{n-1}$ for $J_n$. 
The case $r=0$ of the following assertion uses the convention that
$L_0=H_{n-1}$ and $L_{-1}=H_n$.

\begin{prop}\label{Gtwisting}
Suppose that $G=G_0\times\mu_p^n$ acts nicely on $S_n$ and 
(via $G\to\mu_p^n$) on $\{ H_1,\dots, H_n\}$.
Then $G$ acts nicely on each $Y_r$, and on its line bundles
$\{ H_1,\dots,H_{n-2}, L_r, L_{r-1}\}$.
\end{prop}

\begin{proof}
The question being local, we may replace $S'$ by a fixed point $s\in S'$,
and $G$ by $\mu_p^n$. We proceed by induction on $r$, 
the case $r=1$ being Example \ref{ex:niceY}, since $L_1=H_n\oo\bLbox$.
Inductively, suppose that $G$ acts nicely on $Y_r$ and on
$\{ H_1,\dots,H_{n-2},L_r,L_{r-1}\}$. Thus there is a factor of $G$ 
isomorphic to $\mu_p$ which acts nontrivially on $L_r$ but acts trivially on
$\{ H_1,\dots,H_{n-2},L_r\}$. Hence this factor acts trivially on
$Y_{r+1} = P_{n-1}(Y_r;H_1,\dots,H_{n-2},L_r)$ and its line bundle
$\bL^\boxtimes$, and nontrivially on $L_{r+1} = L_{r-1}\oo\bL^\boxtimes$. 
The assertion follows.  
\end{proof}

\begin{cor}\label{GonSJ}
$G=\mu_p^n$ acts nicely on $(S,J)$. 
\end{cor}

\begin{proof}
By Definition \ref{Ytower}, $S_{n-1}=Y_p$, $J_{n-1}=L_p$ 
and $J'_{n-1}=L_{p-1}$.
By \ref{Gtwisting} with $r=p$, 
$G$ acts nicely on $S_{n-1}$ and on
$\{ H_1,\dots,H_{n-2}, J_{n-1}, J'_{n-1}\}$. By downward induction,
$G=\mu_p^{n-i}\times\mu_p^i$ acts nicely on $S_i$ and 
$\{ H_1,\dots,H_{i-1},J_i,J'_i \}$ for all $i\le n$. The case $i=1$ is the
conclusion, since $(S,J)=(S_1,J_1)$.
\end{proof}

\begin{subrem}
If $G=\mu_p^n$ acts nicely on $S'$, Rost \cite[p.2]{Rost-CSV} would say that 
a fixed point $s\in S'$ is {\it twisting} for $\{ H_1,\dots,H_r\}$ 
if the map
$G\to \mu_p^r\subset\prod k(s)^\times=\prod\Aut(H_i|_s)$ is a surjection.
\end{subrem}

\section{$G$-fixed point equivalences}

Let $\cA=\cA(J)$ be the Kummer algebra over the variety $S$ of the Chain Lemma
\ref{thm:chainlemma}, as in \ref{Kummeralgebra}. 
The group $G=\mu_p^n$ acts nicely on $S$ and $J$ by \ref{GonSJ}, and
on $\cA$ and $\bP(\cA)$ by \ref{ex:GKummer}. In this section, we introduce
two $G$-varieties $\barT$ and $Q$, parametrized by norm conditions,
and show that they are $G$-fixed point equivalent to $\bP(\cA)$ and
$\bP(\cA)^p$, respectively. This will be used in the next section
to show that $\barT$ is $G$-fixed point equivalent to the Weil
restriction of $Q_E$ for any Kummer extension $E$ of $k$.

We begin by defining fixed point equivalence and the variety $Q$.

\begin{defn}\label{def:fpe}
Let $G$ be an algebraic group. We say that two $G$-varieties $X$ and $Y$
are {\it $G$-fixed point equivalent} if $\Fix_GX$ and $\Fix_GY$ are
0-dimensional, lie in the smooth locus of $X$ and $Y$, and there is a
separable extension $K$ of $k$ and a bijection $\Fix_G(X_K)\to\Fix_G(Y_K)$
under which the families of tangent spaces at the fixed points are 
isomorphic as $G$-representations over $K$. 
\end{defn}

\begin{defn}\label{def:Q}
Recall from \ref{Kummeralgebra} that the norm $\cA\map{N}\cO_S$ is
equivariant, and homogeneous of degree $p$.
We define the $G$-variety $Q$ over $S\times\A^1$, and its fiber
$Q_w$ over $w\in k$,
by the equation $N(\beta)=w$:
\begin{align*}
Q =&  ~\{ [\beta,t]\in \bP(\cA\oplus\cO)\times\A^1: N(\beta)=t^pw \}, \\
Q_w =& ~\{ [\beta,t]\in \bP(\cA\oplus\cO): N(\beta)=t^pw \},
\quad \text{ for } w\in k.
\end{align*}
\end{defn}

Since $\dim(S)=p^n-p$ we have $\dim(Q_w)=p^n-1$. If $w\neq 0$, then it is proved in \cite[\S 2]{SJ} that $Q_w$ is geometrically irreducible and that the open subscheme where $t\neq 0$ is smooth. 

If $w\ne0$, $Q_w$ is disjoint from the section $\sigma:S\cong\bP(\cO)\to\bP(\cA\oplus\cO)$; 
over each point of $S$, the point $(0:1)$ is disjoint from $Q_w$.  
Hence the projection $\bP(\cA\oplus\cO)-\sigma(S)\to\bP(\cA)$ 
from these points induces an equivariant morphism $\pi: Q_w\to Y=\bP(\cA)$, 
$\pi(\beta,t)=\beta$. This is a cover of degree $p$ over its image, since $\pi(\beta,t)=\pi(\beta,\zeta t)$
for all $\zeta\in\mu_p$.

\begin{thm}\label{thm:Xb}
If $w\ne0$, $G$ acts nicely on $Q_w$ and
$\Fix_G Q_w \cap (Q_w)_{\textrm{sing}} = \emptyset$. Moreover,
$Q_w$ and $Y=\bP(\cA)$ are $G$-fixed point equivalent over the field
$\ell = k(\root{p}\of b)$.
\end{thm}

\begin{proof} 
Since the maps $Q_w\map{\pi}Y\to S$ are equivariant, $\pi$ maps 
$\Fix_GQ_w$ to 
$\Fix_GY$, and both lie over the finite set $\Fix_GS$ of $k$-rational points.
Since the tangent space $T_y$ is the product of $T_sS$ and the tangent
space of the fiber $Y_s$, and similarly for $Q_w$, it suffices to
consider a $G$-fixed point $s\in S$.

By \ref{Gtwisting} and Lemma \ref{lem:Glines}, $G$ acts nontrivially 
on $L=J|_s$ via a projection $G\to\mu_p$. By Example \ref{ex:GKummer},
$G$ acts nicely on $\bP(\cA)$.
Thus there is no harm in assuming that $G=\mu_p$ and that $L$ is the
standard 1-dimensional representation.

Let $y\in Y$ be a $G$-fixed point lying over $s$. 
By \ref{ex:Gprojective}, the tangent space of $Y|_s$ at $y$ is the reduced
regular representation, and $y$ is one of $[1]$, $[L]$, \dots $[L^{p-1}]$.

We saw in Example \ref{ex:GB} that a fixed point 
$[\,a_0:a_1:\cdots:a_{p-1}:t\,]$
of $G$ in $\bP(\cA\oplus\cO)|_s$ is either one of the points
$e_i=[\,\cdots0:a_i:0\cdots:0\,]$, which do not lie on $Q_w$, 
or a point on the projective line $\{[\,a_0:0:t\,]\}$. 
By inspection, $Q_w\otimes_k\ell$ 
meets the projective line in the $\ell$-points
$[\,\zeta{\root{p}\of b}:0:\cdots:0:1\,]$, $\zeta\in\mu_p$.
Each of these $p$ points is smooth on $Q_w$, and the tangent space (over $s$)
is the reduced regular representation of $G$.
\end{proof}

\begin{subrem}
Since $\pi([\,\zeta{\root{p}\of b}:0:\cdots:0:1\,])=[1]$ for all $\zeta$,
$\Fix_G(Q_w)\map{\pi}\Fix_G(Y)$ is {\it not} a scheme isomorphism over $\ell$.
\end{subrem}

\begin{rem}\label{cor:NV}
For any $w\in k^\times$ of $N$, any desingularization $Q'$ 
of $Q_w$ is a smooth, geometrically irreducible splitting variety for the symbol 
$\{ a_1,...,a_{n},w \}$ in $K^M_{n+1}(k)/p$. 

Assuming the Bloch-Kato conjecture for $n$, 
Suslin and Joukhovitski show it is a norm variety
in \cite[\S2]{SJ}. Note that the variety $X_w$ of \ref{exist:NV} is 
birationally a cover of $Q_w$.
\end{rem}
\goodbreak

To construct $\barT$, we fix a Kummer extension $E=k(\epsilon)$ of $k$.
Let $\cB$ be the $\cO_S$-subbundle 
$(\cA\!\otimes\!1) \oplus (\cO_S\!\otimes\!\epsilon)$ of
$\cA_E=\cA\otimes_kE$ and let $N_{\cB}:\cB\to\cO_S\otimes_kE$ be 
the map induced by the norm on $\cA_E$. 

\begin{defn}
Let $U$ be the variety $\bP(\cA)\times\bP(\cB)^{\times (p-1)}$
over $S^{\times p}$, and let $L$ be the line bundle
$\bL(\cA)\boxtimes \bL(\cB)^{\boxtimes (p-1)}$ over $U$, given as 
the exterior product of the tautological bundles.
The product of the various norms defines an algebraic morphism 
$N:L \to \cO_S\!\otimes\!{E}$. 
\end{defn}

\begin{lem}\label{lem:nosplitalgebras}
Let $u\in U$ be a point over $(s_0,s_1,\dotsc,s_{p-1})$, and
write $A_i$ for the $k(s_i)$-algebra $\cA|_{s_i}$. 
Then the following hold. 
\begin{enumerate}
\item If $\syma$ doesn't split at any of the points $s_0,\dotsc,s_{p-1}$,
then the norm map $N: L_u \to k(u)\otimes E$ 
is non-zero.
\item If $\syma|_{s_0}\ne0$ in $K^M_n(k(s_0))/p$, then $A_0$ is a field.
\item For $i\geq 1$, if $\syma|_{E(s_i)}\ne0$ in $K^M_n(E(s_i))/p$ then  
$A_i\otimes E$ is a field.
\end{enumerate}
\end{lem}

\begin{proof}
The first assertion follows from part (4) of the Chain Lemma 
\ref{thm:chainlemma}, since by \ref{Kummeralgebra} the norm on $L$
is induced from the $p$-form $\gamma_1$ on $J$.
Assertions (2--3) follow from part (2) of the Chain Lemma, since
$\syma\ne0$ implies that $\gamma$ is nontrivial.
\end{proof}

\begin{defn}\label{def:barT}
Let $\A^E$ denote the Weil restriction $\mathrm{Res}_{E/k}\A^1$, 
characterized by $\A^E(F) = F\otimes_k E$ (\cite{Weil}).
Let $\barT$ denote the subvariety of $\bP(L\oplus\cO)\times\A^E$ 
consisting of all points $([\alpha:t],w)$ such that $N(\alpha)=t^pw$ in $E$. 
We write $\barT_w$ for the fiber over a point $w\in\A^E$.
Note that $\dim(\barT_w)=p^{n+1}-p=p~\dim(Q_w)$.
\end{defn}

\begin{notate}\label{rem:pointsonT}
Let $([\alpha:t],w)$ be a $k$-rational point on $\barT$, so that
$w\in\A^E(k)=E$. We may regard
$[\alpha:t]\in \bP(L\oplus\cO)(k)$ as being
given by a point $u\in U(k)$, lying over a point 
$(s_0,\dotsc,s_{p-1})\in S(k)^{\times p}$, and a nonzero pair 
$(\alpha,t)\in L_{u}\times k$ (up to scalars). 
From the definition of $L$, we see that (up to scalars) 
$\alpha$ determines a $p$-tuple 
$(b_0,b_1+t_1\epsilon,\dots,b_{p-1}+t_{p-1}\epsilon)$,
where $b_i\in \cA|_{s_i}$ and $t_i\in k$.
When $\alpha\ne0$, $b_0\neq 0$ and for all $i>0$, $b_i\ne0$ or $t_i\ne0$.
Finally, writing $A_i$ for $\cA|_{s_i}$, the norm condition says that in $E$:
\[N_{A_0/k}(b_0)\prod\nolimits_{i = 1}^{p-1} 
N_{A_i\otimes E/E} (b_i + t_i\epsilon) = t^p w.
\]
If $k\subseteq F$ is a field extension, then an $F$-point of
$\barT$ is described as above, replacing $k$ by $F$ and $E$ by
$E\otimes_k F$ everywhere.
\end{notate} 

\begin{subrem}\label{rem:morepointsonT} 
If $w\neq 0$, then $\alpha\neq 0$, because $N(\alpha)=t^pw$
and $(\alpha,t)\ne(0,0)$.
\end{subrem}

\begin{lem}\label{lem:tnotzero}
If $\barT$ has a $k$-point with $t=0$ then 
$\syma|_E=0$ in $K^M_n(E)/p$.
\end{lem}

\begin{proof}
We use the description of a $k$-point of $\barT$ from 
\ref{rem:pointsonT}. If $t = 0$, then $\alpha\neq 0$, therefore
$b_0\neq 0\in A_0$ and $b_i + t_i\epsilon \neq 0\in A_i\otimes E$. 
By Lemma \ref{lem:nosplitalgebras}, if $\syma|_E \neq 0$ in $K^M_n(E)/p$ 
then $A_0$ and all the algebras $A_i\otimes E$ are fields, so that 
$N(\alpha) = N_{A_0/k}(b_0)\prod_{i=1}^{p-1}
N_{A_i\otimes E/E}(b_i + t_i\epsilon)\neq 0$, a contradiction to $t^pw=0$.
\end{proof}

Consider the projection $\barT\to\A^E$ onto the second factor, and write
$\barT_w$ for the (scheme-theoretic) fiber over $w\in\A^E$.
Combining \ref{lem:nosplitalgebras} with \ref{lem:tnotzero} we obtain
the following consequence (in the notation of \ref{rem:pointsonT}):

\begin{cor}\label{cor:Theorem5}
If $\syma\ne0$ in $K^M_n(E)/p$ and $w\ne0$ is such that $\barT_w$ 
has a $k$-point, 
then $A_0$ and the $A_i\otimes E$ are fields and $w$ is a product of
norms of an element of $A_0$ and elements in the subsets 
$A_i+\epsilon$ of $A_i\otimes_kE$.
\end{cor}

\begin{subrem}\label{rem:kpointgivesNP}
In Theorem \ref{Theorem6} we will see that if $w$ is a generic element 
of $E$ then such a $k$-point exists. 
\end{subrem}

The group $G = \mu_p^n$ acts nicely on $S$ and $J$ by \ref{GonSJ}, and
on $\cA$ and $\bP(\cA)$ by \ref{ex:GKummer}. It acts trivially on $\A^E$,
so $G$ acts on $B$, $U$ and $\barT$ (but not nicely; see \ref{rem:Gproduct}).

In the notation of \ref{rem:pointsonT},
if $([\alpha:t],w)$ is a fixed point of the $G$-action on $\barT$
then the points $u_0\in \bP(\cA)$ and $s_i\in S$ are fixed, and therefore 
are $k$-rational (see \ref{def:Gnice}). 
If $u$ is defined over $F$, each point $(b_i:t_i)$ is fixed in $\cB|_{s_i}$.
Since $S$ acts nicely on $J$, Example \ref{ex:GB} shows that if $t=0$ then
either $t_i\ne0$ (and $b_i\in F\subset A_i\otimes F$) or else $t_i=0$
and $0\ne b_i\in J|_{s_i}^{\otimes r_i}\otimes F \subseteq A_i\otimes F$ 
is for some $r_i$, $0\le r_i<p$.

\begin{lem}\label{lem:Tfix1}
For all $w$, $\Fix_G\barT_w$ is disjoint from the locus where $t=0$.
\end{lem}

\begin{proof}
Suppose $([\alpha:0],w)$ is a fixed point defined over a field $F$
containing $k$. As explained above, 
$b_0\ne0$ and (for each $i>0$) $b_i+t_i\epsilon\ne0$ and either $t_i\ne0$ 
or there is an $r_i$ so that $b_i\in J^{r_i}|_{s_i}\otimes F$. 
Let $I$ be the set of indices such that $t_i\neq0$. 

By Example \ref{ex:GKummer}, 
$b_0\in J|_{s_0}^{\otimes r_0}$ for some $r_0$,
and hence $N_{A_0}(b_0)$ is a unit in $k$, 
because the $p$-form $\gamma$ is nontrivial on $J|_{s_0}$. 
Likewise, if $i\notin I$, then 
$N_{A_i\otimes F/F}(b_i)$ is a unit in $F$.

Now suppose $i\in I$, \ie $t_i\ne0$, and
recall that in this case $b_i\in F\subset\! A_i\otimes\! F$. 
If we write $EF$ for the algebra
$E\otimes F \cong F[\epsilon]/(\epsilon^p - e)$, then the norm 
from $A_i\otimes EF$ to $EF$ is simply the $p$-th power
on elements in $EF$, so
$N_{A_i\otimes EF/EF}(b_i + t_i\epsilon) = (b_i + t_i\epsilon)^p$ 
as an element in the algebra $EF$. Taking the product, and keeping in mind
$t = 0$, we get the equation
\[
\prod\nolimits_{i\in I} N_{A_i\otimes EF/EF}(b_i + t_i\epsilon)
 = \prod\nolimits_{i\in I}(b_i + t_i\epsilon)^p = 0.\]

Because $EF$ is a separable $F$-algebra, it has no
nilpotent elements. We conclude that
\[
\prod\nolimits_{i\in I}(b_i + t_i\epsilon) = 0.
\] 
The left hand side of this equation is a polynomial of degree at most
$p-1$ in $\epsilon$; since $\{1,\epsilon,\dotsc,\epsilon^{p-1}\}$ is a
basis of $F\otimes E$ over $F$, that polynomial must be zero.
This implies that $b_i = t_i = 0$ for some $i$, a contradiction.
\end{proof}

\begin{prop}\label{prop:thm6(1)}
If $w\in\A^E$ is generic then 
$\Fix_G\barT_w$ lies in the open subvariety where
$t\prod_{i = 1}^p t_i \neq 0$. 
\end{prop}

\begin{subrem}\label{rem:w-fiber}
The open subvariety in \ref{prop:thm6(1)} is $G$-isomorphic 
(by setting $t$ and all
$t_i$ to $1$) to a closed subvariety of $\A(\cA)^p$, namely the fiber
over $w$ of the map $N_{\cA\otimes E/E}: \A(\cA)^p\to\A^E$ defined by
\[
N(b_0,\dots,b_{p-1}) =  N_{A_0/k}(b_0)
\prod\nolimits_1^{p-1}N_{A_i\otimes E/E}(b_i+\epsilon).
\]
Indeed, $\A(\cA)^p$ is $G$-isomorphic to an open subvariety of $\barT$
and $N_{A_i\otimes E/E}$ is the restriction of $\alpha\mapsto N(\alpha)$.
\end{subrem}

\begin{proof}
By Lemma \ref{lem:Tfix1}, $\Fix_G\barT_w$ is disjoint from the locus
where $t=0$, so we may assume that $t=1$. Since $w$ is generic, we may
also take $w\ne0$. So let $([\alpha :1],w)$ be a fixed point defined 
over $F\supseteq k$ for which $t_j=0$. 
As in the proof of the previous
lemma, we collect those indices $i$ such that $t_i\neq 0$ into a set
$I$, and write $EF$ for $E\otimes_kF$. 
Recall that for $i\in I$, we have $b_i\in F$. Since $j\notin I$,
we have that $\lvert I\rvert \leq p-2.$ For $i\notin I$,
\[
N_{A_i\otimes EF/EF}(b_i + t_i\epsilon) = N_{A_i\otimes F/F}(b_i)\in F^\times
\]
(the norm cannot be $0$ as $t^p w = w\neq 0$ by assumption).
So we get that 
\[ \prod\nolimits_{i\in I} (b_i + t_i\epsilon)^p = \xi w \]
for some $\xi\in F^\times$. If we view $\xi w$ as a point in
$\bP(E)(F) = (EF-\{0\})/F^\times$, 
then we get an equation of the form
\[
 \left[\prod\nolimits_{i\in I} (b_i + t_i\epsilon)^p\right] = [w].
\]
But the left-hand side lies in the image of the morphism
$\prod_{i\in I}\bP^1 \to \bP(E)$ which sends $[b_i:t_i]\in
\bP^1(F)$ to $[\prod (b_i + t_i\epsilon)^p]\in \bP(E)(F).$
Since $\lvert I\rvert\leq p-2$, this image is 
a proper closed subvariety, proving the assertion for generic $w$.
\end{proof}

\begin{thm}\label{Theorem6}
For a generic closed point $w\in \A^E$, $\barT_w$ is $G$-fixed point 
equivalent to the disjoint union of $(p-1)!$ copies of $\bP(\cA)^p$
\end{thm}

\begin{proof}
Since both lie over $S$, it suffices to consider a $G$-fixed point 
$s=(s_0,\dotsc,s_{p-1})$ in $S(k)^p$ 
and prove the assertion for the fixed points over $s$. Because
$G$ acts nicely on $S$ and $J$, $k(s)=k$ and (by Lemma \ref{lem:Glines}) 
$G$ acts on $J_{s}$ via a projection $G\to\mu_p$ as the standard
representation of $\mu_p$. Note that $J_{s}=J_{s_i}$ for all $i$.

By Example \ref{ex:GKummer}, there are precisely $p$ fixed points on 
$\bP(\cA)$ lying over a given fixed point $s_i\in S(k)$, and at each of 
these points the (vertical) tangent space is the reduced regular 
representation of $\mu_p$. Thus each fixed point in $\bP(\cA)^p$ is
$k$-rational, the number of fixed points over $s$ is $p^p$, 
and each of their tangent spaces is the
sum of $p$ copies of the reduced regular representation.

Since $w$ is generic, we saw in \ref{prop:thm6(1)} that all the fixed points of
$\barT_w$ satisfy $t\neq 0$ and $t_i\neq 0$ for $1\leq i\leq p-1$.
By Remark \ref{rem:w-fiber}, they lie  in the
affine open $\A(\cA)^p$ of $\bP(L\oplus\cO)$. Because $\mu_p$ acts nicely on
$J_s$, an $F$-point $b=(b_0,\dots,b_{p-1})$ of $\A(\cA)^p$ is fixed 
if and only if each $b_i\in F$. That is, $\Fix_G(\A(\cA)^p)=\A^p$.
Now the norm map restricted to the fixed-point set is just the map
$\A^p\to\A^E$ sending $b$ to $b_0^p\prod_{i=1}^{p-1}(b_i + \epsilon)^p$.
This map is finite of degree $p^p(p-1)!$, and \'etale for generic $w$,
so $\Fix_G(\barT_w)$ has $p^p(p-1)!$ geometric points for generic $w$. 
This is the same number as the fixed points
in $(p-1)!$ copies of $\bP(\cA)$ over $s$, so it suffices to check their
tangent space representations.

At each fixed point $b$, the tangent space of $\A(\cA)^p$ (or $\barT$) 
is the sum of $p$ copies of the regular representation of $\mu_p$. Since
this tangent space is also the sum of the tangent space of $\A^p$ 
(a trivial representation of $G$) and the normal bundle of $\A^p$ in 
$\barT$, the normal bundle must then be $p$ copies of the reduced 
regular representation of $\mu_p$. Since the tangent space
of $\A^p$ maps isomorphically onto the tangent space of $\A^E$ at $w$,
the tangent space of $\barT_w$ is the same as the normal bundle of 
$\A^p$ in $\barT$, as required.
\end{proof}

\begin{subrem}
The fixed points in $\barT_w$ are not necessarily rational points,
and we only know that the isomorphism of the tangent spaces at the fixed 
points holds on a separable extension of $k$. This is parallel to the
situation with the fixed points in $Q_w$ described in Theorem \ref{thm:Xb}.
\end{subrem}

\section{A $\nu_{n}$-variety.}\label{sec:bpatheorem}

The following result will be needed in the proof of the norm principle.

\begin{thm}\label{thm:toddPA}
Let $S$ be the variety of the chain lemma for some symbol $\syma\in K_n^M(k)/p$ and $\cA = \bigoplus_{i=0}^{p-1} J^{\otimes i}$ the sheaf of Kummer algebras over $S$. Then the projective bundle $\bP(\cA)$ has dimension $d = p^n - 1$ and $p^2\nmid s_d(\bP(\cA)).$ 
\end{thm}

\begin{proof}
Let $\pi: \bP(\cA) \to S$ be the projection. The statement about the dimension is trivial.
In the Grothendieck group $K_0(\bP(\cA))$, we have that 
\[
[T_{\bP(\cA)}] = \pi^*([T_S]) + [T_{\bP(\cA)/S}]
\] 
where $T_{\bP(\cA)/S}$ is the relative tangent bundle. The class $s_d$ is additive, 
and the dimension of $S$ is less than $d$, so we conclude that $s_d(\bP(\cA)) = s_d(T_{\bP(\cA)/S}).$ 
Now $[T_{\bP(\cA)/S}] = [\pi^*(\cA)\otimes \cO(1)_{\bP(\cA)/S}] - 1$; applying additivity again, together with the definition of $s_d$ and the decomposition of $\cA$ and hence $\pi^*(\cA)$ into line bundles, we obtain
\[ s_d(\bP(\cA)) = \deg \sum_{i = 0}^{p-1} c_1(\pi^*J^{\otimes i}\otimes \cO(1))^d.\]
The projective bundle formula presents the Chow ring $CH^*(\bP(\cA))$ as: 
\[ CH^*(\bP(\cA)) = CH^*(S)[y]/(\prod_{i = 0}^{p-1}(y - ix)) \]
where $x = -c_1(J)\in CH^1(S)$ and $y = c_1(\cO(1))\in CH^1(\bP(\cA)).$
Then $s_d(\bP(\cA))$ is the degree of the following element of the ring $CH^*(\bP(\cA))$:
\[
s'_d(\bP(\cA)) = \sum\nolimits_{i = 0}^{p-1} (y - ix)^d = \sum\nolimits_{i = 0}^{p-1} a_i y^i x^{d-i}
\]
for some integer coefficients $a_i$. Since $x\in CH^1(S)$, we have $x^r=0$ for any $r > \dim(S) = p^n - p$.
It follows that $s'_d(\bP(\cA)) = a_{p-1} y^{p-1} x^{\dim(S)}$. By part (6) of the Chain Lemma \ref{thm:chainlemma}, 
the degree of $x^{\dim(S)} = (-1)^{\dim(S)} c_1(J)^{\dim(S)}$ is prime to $p$. 
In addition, $\pi_*(y^{p-1}) = \pi_*(c_1(\cO(1))^{p-1}) = [S]\in CH^0(S)$. By the projection formula
$s_d(\bP(\cA))=a_{p-1}\deg\, x^{\dim(S)}$. Thus to prove the theorem, it suffices to show that 
$a_{p-1} \equiv p \pmod{p^2}$; this algebraic calculation is achieved in Lemma \ref{lem:coefficient} below.
\end{proof}

\begin{lem}\label{lem:coefficient}
In the ring $R = \Z/p^2 [x,y]/(\prod_{i=0}^{p-1} (y - ix))$, the coefficient of $y^{p-1}$ in 
$u_m = \sum_{i = 0}^{p-1} (y - ix)^{p^m\!-1}$ is $px^b$, with $b = p^m - p$.
\end{lem}

\begin{proof}
Since $u_m$ is homogeneous of degree $p^m-1$, it suffices to determine the coefficient of $y^{p-1}$ in $u_m$ in the ring 
\[R/(x-1) = \Z/p^2[y]/(\prod\nolimits_{i=0}^{p-1} (y-i)) \cong \prod\nolimits_{i=0}^{p-1} \Z/p^2.\]
If $m=1$, then $u_1 = \sum_{i=0}^{p-1} (y-i)^{p-1}$ is a polynomial of degree $p-1$ with leading term $py^{p-1}$. 
Inductively, we use the fact that for all $a\in \Z/p^2$, we have
\[
a^{p^2-p} = \begin{cases} 0,& \text{if~}p\mid a \\ 1,& \text{else}. \end{cases}
\]
Thus for $m\geq 2$, if we set $k = (p^{m-1} - 1)/(p-1)$, then $a^{p^m-1} = a^{(p-1) + k(p^2-p)} = a^{p-1}\in \Z/p^2$, and therefore
\[ u_m = \sum_{i=0}^{p-1} (y-i)^{p^m-1} = \sum_{i=0}^{p-1} (y-i)^{p-1} = u_1\]
holds in $R/(x-1)$; the result follows. 
\end{proof}

\section{The Norm Principle}\label{sec:NP}

We now turn to the Norm Principle, which concerns the group $A_0(X,\cK_1)$
associated to a variety $X$. In the literature, this group is also known as
$H_{-1,-1}(X)$ and $H^d(X,\cK_{d+1})$, where $d=\dim(X)$. We recall the
definition from \ref{def:A0K1}.

\begin{defn}\label{def:A1}
If $X$ is a regular scheme then $A_0(X,\cK_1)$ is the cokernel of the map
$\oplus_{y}K_2(k(y)) \map{(\partial_{xy})} \oplus_{x} k(x)^\times$.
In this expression, the first sum is taken over all points $y\in X$ of 
dimension~1, and the second sum is over all closed points $x\in X$.
The map $\partial_{xy}:K_2(k(y))\to k(x)^\times$ is the tame symbol 
associated to the discrete valuation on $k(y)$ associated to $x$;
if $x$ is not a specialization of $y$ then $\partial_{xy}=0$.
If $x\in X$ is closed and $\alpha\in k(x)^\times$ we write $[x,\alpha]$
for the image of $\alpha$ in $A_0(X,\cK_1)$.
\end{defn}

The group $A_0(X,\cK_1)$ is covariant for proper morphisms $X\to Y$, and 
clearly $A_0(\Spec k,\cK_1)=k^\times$ for every field $k$. Thus if 
$X\to\Spec(k)$ is proper then there is a morphism 
$N:A_0(X,\cK_1)\to k^\times$, whose restriction to the group of
units of a closed point $x$ is the norm map $k(x)^\times\to k^\times$.
That is, $N[x,\alpha]=N_{k(x)/k}(\alpha)$.

\begin{defn}
When $X$ is smooth and proper over $k$, we write $\oA(X,\cK_1)$ for the
quotient of $A_0(X,\cK_1)$ by the relation that 
$[x_1,N_{x/x_1}(\alpha)]=[x_2,N_{x/x_2}(\alpha)]$ for every closed point 
$x=(x_1,x_2)$ of $X\times_kX$ and every $\alpha\in k(x)^\times$. 
\end{defn}

It is proven in \cite[1.5--1.7]{SJ} that if $X$ has a $k$-rational
point then $\oA(X,\cK_1)=k^\times$; if $X(k)=\emptyset$, then both the kernel
and cokernel of $N:\oA(X,\cK_1)\to k^\times$ have exponent $n$, where
$n$ is the gcd of the degrees $[k(x):k]$ for closed $x\in X$. In
addition, if $x, x'$ are two points of $X$ then for any field map
$k(x')\to k(x)$ over $k$ and any $\alpha\in k(x)^\times$ we have
$[x,\alpha]=[x',N_{x/x'}\alpha]$ in $\oA(X,\cK_1)$.

To illustrate the advantage of passing to $\oA$, consider a cyclic field 
extension $E/k$. Then 
$A_0(\Spec E,\cK_1)=E^\times$ and by Hilbert 90, there is an exact sequence
$$ 0 \to\oA(\Spec E,\cK_1) \to k^\times \to\text{Br}(K/k)\to0.$$

We now suppose that $k$ is a $p$-special field, so that the kernel and
cokernel of $N:\oA(X,\cK_1)\to k^\times$ are $p$-groups, and that
$X$ is a Norm variety (a $p$-generic splitting variety of dimension
$p^n-1$).
The Norm Principle is concerned with reducing the degrees of the field 
extensions $k(x)$ used to represent elements of $\oA(X,\cK_1)$.
For this, the following definition is useful.

\begin{defn}
Let $\tA_0(k)$ denote the subset of elements $\theta$ of 
$\oA(X,\cK_1)$ represented by $[x,\alpha]$ where $k(x)=k$
or $[k(x):k]=p$. If $E/k$ is a field extension, $\tA_0(E)$ denotes the 
corresponding subset of $\oA(X_E,\cK_1)$.
\end{defn}

\begin{lem}\label{tAsubgroup} 
If $k$ is $p$-special and $X$ is a Norm variety, then
$\tA_0(k)$ is a subgroup of $\overline{A}_0(X,\cK_1)$.
\end{lem}

\begin{proof}
By the Multiplication Principle \cite[5.7]{SJ}, which depends upon the
Chain Lemma \ref{thm:chainlemma}, we know that for each $[x,\alpha]$, 
$[x',\alpha']$ in $\tA_0(k)$, there is a 
$[x'',\alpha'']\in\tA_0(k)$ so that
$[x,\alpha]+[x',\alpha']=[x'',\alpha'']$ in $\overline{A}_0(X,\cK_1)$.
Hence $\tA_0(k)$ is closed under addition. It is nonempty because
$E=k[\root{p}\of{a_1}]$ splits the symbol and therefore $X(E)\ne\emptyset$.
It is a subgroup because $[x,\alpha]+[x,\alpha^{-1}]=[x,1]=0$.
\end{proof}

\begin{lem}[{\cite[1.24]{SJ}}]\label{lem:Nspecial} 
If $k$ is $p$-special and $X$ is a Norm variety,
the restriction of $\overline{A}_0(X,\cK_1)\map{N} k^\times$
to $\tA_0(k)$ is an injection.
\end{lem}

\begin{proof} 
Let $[x,\alpha]$ represent $\theta\in\tA_0(k)$. If
$N(\theta)=N_{k(x)/k}(\alpha)=1$ then $\alpha=\sigma(\beta)/\beta$.
for some $\beta$ by Hilbert's Theorem 90. But 
$[x,\sigma(\beta)]=[x,\beta]$ in $\tA_0(k)$; see \cite[1.5]{SJ}.
\end{proof}

\begin{subex}\label{ex:trivialA0}
If $X$ has a $k$-point $z$, then the norm map $N$ of \ref{def:A0K1} is an 
isomorphism $\tA_0(k)\cong\oA(X,\cK_1)\map{\simeq} k^\times$, split by
$\alpha\mapsto[z,\alpha]$. Indeed, for every closed point $x$ of $X$
we have $[x,\alpha]=[z,N_{k(x)/k}\alpha]$ in $\oA(X,\cK_1)$,
by \cite[1.5]{SJ}.
\end{subex}

Our goal in the next section is to prove the following theorem.
Let $E/k$ be a field extension with $[E:k]=p$. Since $k$ has $p$th 
roots of unity, we can write $E=k(\epsilon)$ with $\epsilon^p\in k$.

\begin{thm}\label{Theorem5}
Suppose that $k$ is $p$-special, $\syma_E\ne0$ and that 
$X$ is a Norm variety for $\syma$. 
For $[z,\alpha]\in\tA_0(E)$, there exist points $x_i\in X$ of 
degree $p$ over $k$, $t_i\in k$ and $b_i\in k(x_i)$
such that $N_{E(z)/E}(\alpha) = \prod N_{E(x_i)/E}(b_i+t_i\epsilon)$.
\end{thm}

Theorem \ref{Theorem5} is the key ingredient in the proof of 
Theorem \ref{Theorem4}.

\begin{thm}\label{Theorem4}
If $k$ is $p$-special and $[E:k]=p$ then
$\overline{A}_0(X_E,\cK_1) \map{N_{E/k}} \overline{A}_0(X,\cK_1)$
sends $\tA_0(E)$ to $\tA_0(k)$.
\end{thm}

\begin{proof} 
If $\syma_E=0$ then the generic splitting variety $X$ has an $E$-point $x$, and
Theorem \ref{Theorem4} is immediate from Example \ref{ex:trivialA0}. 
Indeed, in this case $X_E$ has an $E$-point $x'$ over $x$,
every element of $\tA_0(E)\cong E^\times$ has the form $[x',\alpha]$,
and $N_{E/k}[x',\alpha]=[x,\alpha]$.
Hence we may assume that $\syma_E\ne0$. 
This has the advantage that $E(x_i)=E\otimes_kk(x_i)$ is a field for every $x_i\in X$.

Choose $\theta=[z,\alpha]\in\tA_0(E)$ and let $x_i\in X$, $t_i$ and $b_i$
be the data given by Theorem \ref{Theorem5}. Each $x_i$ lifts to an
$E(x_i)$-point $x_i\otimes E$ of $X_E$ so we may consider the element
\[ \theta' = \theta - \sum [x_i\otimes E, b_i+t_i\epsilon]
\in \overline{A}_0(X_E,\cK_1).
\]
By \ref{tAsubgroup} over $E$, $\theta'$ belongs to the subgroup $\tA_0(E)$. 
By Theorem \ref{Theorem5}, its norm is
\[ 
N(\theta') = N_{E(z)/E}(\alpha)/\prod N_{E(x_i)/E}(b_i+t_i\epsilon) = 1.
\]
By Lemma \ref{lem:Nspecial}, $\theta'\!=0$. Hence 
$N_{E/k}(\theta)=\sum \left[x_i,N_{E(x_i)/k(x_i)}(b_i+t_i\epsilon)\right]$
in $\overline{A}_0(X,\cK_1)$.
Since $\tA_0(k)$ is a group by \ref{tAsubgroup}, this is 
an element of $\tA_0(k)$.
\end{proof}

\begin{cor}[Theorem \ref{thm:normvar}(3)]\label{Theorem2}
If $k$ is $p$-special then $\tA_0(k) = \overline{A}_0(X,\cK_1)$,
and $N:\overline{A}_0(X,\cK_1)\to k^\times$ is an injection.
\end{cor}

\begin{proof} We may suppose that $X(k)=\emptyset$.
For every closed $z\in X$ there is an intermediate subfield $E$ with
$[k(z):E]=p$ and a $k(z)$-point $z'$ in $X_E$ over $z$. Since 
$[z',\alpha]\in\tA_0(E)$, Theorem \ref{Theorem4} implies that 
$[z,\alpha]=N[z',\alpha]$ is in $\tA_0(k)$. This proves the first 
assertion. The second follows from this and Lemma \ref{lem:Nspecial}.
\end{proof}

The Norm Principle of the Introduction 
follows from Theorem \ref{Theorem4}. 

\begin{NPproof}
We consider a generator $[z,\alpha]$ of $\overline{A}_0(X,\cK_1)$.
Since $[k(z):k]=p^\nu$ for $\nu>0$, there is a subfield $E$ 
of $k(z)$ with $[k(z):E]=p$, and $z$ lifts to a 
$k(z)$-point $z'$ of $X_E$. By construction, $[z',\alpha]\in\tA_0(E)$ and 
$\overline{A}_0(X_E,\cK_1) \to \overline{A}_0(X,\cK_1)$ 
sends $[z',\alpha]$ to $[z,\alpha]$.
By Theorem \ref{Theorem4}, $[z,\alpha]$ is in $\tA_0(k)$, {\it i.e.},
is represented by an element $[x,\alpha]$ with $[k(x):k]=p$.
\qquad\qed\end{NPproof}

\section{Expressing Norms}

Recall that $E=k(\epsilon)$ is a fixed Kummer extension of a $p$-special
field $k$, and $X$ is a Norm variety over $k$ for the symbol $\syma$. The purpose of this section 
is to prove Theorem \ref{Theorem5}, that if an element $w\in E$ is a norm 
for a Kummer point of $X_E$ then $w$ is a product of norms of 
the form specified in Theorem \ref{Theorem5}.
%

Recall from \ref{def:Q} that $Q \subseteq \bP(\cA\oplus\cO)\times\A^1_k$ is the 
variety of all points $([\beta,t],w)$ such that $N(\beta) = t^p w$, and let 
$q: Q \to \A^1_k$ be the projection. Extending the base field to $E$ and 
applying the Weil restriction functor, we obtain a morphism 
\[\Rq = \mathrm{Res}_{E/k}(q_E): \RQ = \mathrm{Res}_{E/k}(Q_E) \to \A^E.\]
Moreover, choose once and for all a resolution of singularities $\tilde{Q} \to Q$, which is an isomorphism where $t\neq 0$. This is possible since $Q$ is smooth where $t\neq 0$, see \ref{def:Q}. 


\begin{rem} \label{rem:norms}
Since $k$ is $p$-special, so is $E$. As stated in Lemma \ref{lem:Nspecial}, 
the norm map $\tA_0(E) \to E^\times$ is injective; we identify $\tA_0(E)$ with its image. 
Thus $[z,\alpha]\in \tA_0(E)$ is identified with $N_{E(z)/E}(\alpha)\in E^\times$. 
By \cite[Theorem 5.5]{SJ}, there is a point $s\in S$ such that $E(z)=\cA_s\otimes E$; 
Under the correspondence $E(z)\cong\A(\cA)_s(E)$, we identify $\alpha$ with a point
of $\A(\cA)(E)$, lying over $s\in S$. 
Then $N_{E(z)/E}(\alpha) = \Rq([\alpha,1],N(\alpha))$. In other words, 
$\tA_0(E)\subseteq E^\times$ is equal to $q(Q(E)) - \{0\}$. 
\end{rem}

To prove Theorem \ref{Theorem5} it therefore suffices to show that $\barT_w(k)$ 
is non-empty when $w = \Rq([\beta,1],w)$. To do this, we will produce a correspondence 
$Z\to \barT\times_{\A^E} \RQ$ that is dominant and of degree prime to $p$ over $\RQ$. 
We construct the correspondence $Z$ using the Multiplication Principle 
of \cite[5.7]{SJ} in the following form.

\begin{lem}[Multiplication Principle]\label{lem:multprinc}
Let $k$ be a $p$-special field. Then the set of values of the map $N: \A(\cA)(k) \to k$ is a multiplicative subset of $k^\times$.
\end{lem}

\begin{proof}
Given Remark \ref{rem:norms}, this is a consequence of Lemma \ref{tAsubgroup}. 
\end{proof}

\begin{lem}\label{lem:Zexists}
Let $F = k(\barT)$ be the function field. Then there exists a finite extension $L/F$, 
of degree prime to $p$, and a point $\xi\in \RQ (L)$ lying over the generic point of $\A^E$.
\end{lem} 

\begin{proof}
Let $F'$ be the maximal prime-to-$p$ extension of $F$; then the field $EF'=E\otimes_k F'$ is $p$-special. 
We may regard the generic point of $\barT$ as an element in $\barT(F)$. 
Applying the inclusion $F\subset F'$ to this element, followed by the projection $\barT\to\A^E$, 
we obtain an element $\omega$ of $\A^E(F') = EF'$. By \ref{rem:pointsonT}, 
$\omega$ is a product of norms from $\A(\cA)(EF')$. By the Multiplication Principle 
\ref{lem:multprinc}, there exists $\beta\in \A(\cA)(EF')$ such that $N(\beta) = \omega$. 
Now let $\xi$ be the point $([\beta,1],\omega)\in \RQ(F')$. Then $\Rq(\xi) = \omega$ 
and $\xi$ is defined over some finite intermediate extension 
$F\subseteq L\subseteq F'$, with $[L:F]$ prime to $p$.
\end{proof}

Write $\eta_L$ for the point of $\barT(L)$ defined by the inclusion $F\subseteq L$. 
We can now define $\barT \stackrel{f}{\leftarrow} Z \smap{g}\RQ$ to be a 
(smooth, projective) model of $(\eta_L,\xi)\in (\barT\times_{\A^E}\RQ)(L).$

\begin{thm}\label{thm:degree}
The morphism $g: Z\to \RQ$ is proper and dominant (hence onto) and of degree prime to $p$.
\end{thm}

\begin{proof}
Let $\omega\in \A^E$ be the generic point, $k(\omega)$ the function field and 
$E(\omega) = E\otimes k(\omega)$. As degree is a generic notion and invariant under 
extension of the base field, we may replace $\barT \leftarrow Z \rightarrow \RQ$ by 
its basechange along the morphism 
\[
\Spec(E(\omega)) \to \Spec(k(\omega)) \map{\omega}\A^E,
\]
to obtain morphisms $f: Z_{E(\omega)} \to \barT_{E(\omega)}$ and 
$g: Z_{E(\omega)}\to \RQ_{E(\omega)}$. Using the normal basis theorem, we can write 
$E(\omega) = E(\omega_1,\dotsc,\omega_p)$ for transcendentals $\omega_i$ 
that are permuted under the action of the cyclic group $\mathrm{Gal}(E/k)$. 

We will apply the DN Theorem \ref{thm:DN} with base field $k' = E(\omega)$. 
In the notation of Theorem \ref{thm:DN}, we let $r = p$; we write $Y$ for some 
desingularization of $\barT_{E(\omega)}$; we let $X$ be $\tRQ_{E(\omega)}$, 
and we let $W$ be a model for $Z_{E(\omega)}$ mapping to $Y$ and $X$. 
Finally, we let $u_i = \{a_1,\dotsc,a_n,\omega_i\}\in K_{n+1}^M(k')/p$. 

Observe that our base field contains $E$, so 
$\tRQ_{E(\omega)} = \mathrm{Res}_{E/k}(\tilde{Q}_E)\times_{\A^E} E(\omega)$ 
splits as a product $\tRQ_{E(\omega)} = \prod_{i = 1}^p \tilde{Q}_{\omega_i}$, 
where $Q_{\omega_i}$ is the fiber of $Q\to \A^1$ over the point 
$\omega_i\in \A^1(E(\omega)) = E(\omega).$
Therefore we have $X = \prod_{i=1}^p X_i$ where $X_i$ is $\tilde{Q}_{\omega_i}$,
the resolution of singularities of $Q_{\omega_i}$. 
By Remark \ref{cor:NV}, $X_i$ is a smooth, 
geometrically irreducible splitting variety for the symbol $u_i$ of dimension $p^n -1$. 
Thus, hypothesis (1) of the DN Theorem \ref{thm:DN} is satisfied. 

By Theorem \ref{thm:bordloc}, $t_{d,1}(X_i) = t_{d,1}(\bP(\cA))$; by Lemma \ref{lem:todd}, 
we conclude that $s_d(X_i) \equiv v s_d(\bP(\cA)) \pmod{p^2}$ for some unit $v\in \Z/p$. 
Since $s_d(\bP(\cA))\not\equiv0$ by Theorem \ref{thm:toddPA}, we conclude that 
hypothesis (3) of the DN Theorem \ref{thm:DN} is satisfied.

Furthermore, $K = k'(X_1\times \dotsm \times X_{i-1})$ is contained in a 
rational function field over $E$; in fact, the field $E(\omega_j)(Q_{\omega_j})$ 
becomes a rational function field once we adjoin $\root{p}\of{\gamma}$. 
Since $E$ does not split $\syma$, $K$ does not split $\syma$ either. 
It follows that $K$ does not split $u_i =\syma\cup \{\omega_i\}$, 
verifying hypothesis (2) of Theorem \ref{thm:DN}. 

We have now checked the hypotheses (1--3) of Theorem \ref{thm:DN}. 
It remains to check that $X$ and $Y$ are $G$-fixed point equivalent 
up to a prime-to-$p$ factor. In fact, we proved in Theorem \ref{Theorem6} that 
$\barT_{E(\omega)}$ is $G$-fixed point equivalent to $(p-1)!$ copies of $\bP(\cA)^p$, 
hence so is $Y$ (since the fixed points lie in the smooth locus), and in 
Theorem \ref{thm:Xb} that $X_i$ is $G$-fixed point equivalent to $\bP(\cA)$. 
That is, $Y$ is $G$-fixed point equivalent to $(p-1)!$ copies of $X$. 
Therefore the DN Theorem applies to show that $g$ is dominant and of degree prime to $p$, as asserted.
\end{proof}

\begin{proof}[Proof of Theorem \ref{Theorem5}]
We have proved that there is a diagram $\barT\stackrel{f}{\leftarrow} Z\stackrel{g}{\rightarrow} \RQ$ 
such that the degree of $g$ is prime to $p$. By blowing up if necessary we may assume that 
$g: Z \to \RQ$ factors through $\tilde{g}: Z\to \tRQ$, with $\deg(\tilde{g})$ prime to $p$.

Let $[z,\alpha]\in \tA_0(E)$, and set $w = N_{E(z)/E}(\alpha)$.
By Remark \ref{rem:norms}, there exists a point $([\beta,1],w)\in \RQ(k)$.
Lift this to a point in $\tRQ(k)$ (recall that $\tRQ \to \RQ$ is an isomorphism where $t\neq 0$). 
Since $Z \to \tRQ$ is a morphism of smooth projective varieties of degree prime to $p$ and 
$k$ is $p$-special, we can lift $([\beta,1],w)$ to a $k$-point of $Z$, and then apply 
$f: Z \to \barT$ to get a $k$-point in $\barT_w$. By the definition of $\barT$ and 
Corollary \ref{cor:Theorem5}, this means that we can find Kummer extensions $k(x_i)/k$ 
(corresponding to points $s_i\in S$, and determining points $x_i\in X$ because 
$X$ is a $p$-generic splitting variety), elements $b_i\in k(x_i)$ and $t_i\in k$ such that 
$w = \prod_i N_{E(x_i)/E}(b_i + t_i\epsilon)$, as asserted.  
\end{proof}

\bigskip\goodbreak
\newpage

\addtocounter{section}{-10}
\renewcommand{\thesection}{\Alph{section}}
\section{Appendix: The DN Theorem}\label{app:A}

In this appendix, we give a proof of the following Degree theorem, which is
used in the proof of the Norm Principle. Throughout, $k$ will be a fixed 
field of characteristic $0$, $p>2$ will be a prime, $n\ge1$ will be an
integer and we fix $d = p^n - 1$.

Recall from Definition \ref{def:fpe} that if $X$ and $Y$ are $G$-fixed point
equivalent then $\dim(X)=\dim(Y)$, the fixed points are 0-dimensional
and their tangent space representations are isomorphic (over $\bar{k}$).

\begin{thm}[DN Theorem]\label{thm:DN} 
For $r\ge1$, let $u_1,...,u_r$ be symbols in $K^M_{n+1}(k)/p$
and let  $X=\prod_1^r X_i$, where the
$X_i$ are irreducible smooth projective $G$-varieties of
dimension $d=p^n-1$ such that:  
\begin{enumerate}
\item $k(X_i)$ splits $u_i$;
\item $u_i$ is non-zero over $k(X_1\times\cdots\times X_{i-1})$; and
\item $p^2\nmid s_d(X_i)$
\end{enumerate}
\noindent
Let $Y$ be a smooth irreducible projective $G$-variety 
which is $G$-fixed point equivalent to the disjoint union of $m$ copies of $X$, where $p\nmid m$.
Let $F$ be a finite extension of $k(Y)$ of degree prime to $p$, and
$\Spec(F) \to X$ a point, with model $f: W \to X$. Then $f$ is
dominant and of degree prime to $p$.
\end{thm}

\[\xymatrix @R=8mm{
\Spec(F) \ar[r] \ar[d]_{\text{finite}} & 
W \ar[d]_g\ar[dr]^{f~\text{(dominant)}}\\
\Spec(k(Y) \ar[r] &Y&X 
}\]

The proof will use two ingredients: the degree formulas \ref{thm:gendegree}
and \ref{thm:higherdegree} below, due to
Levine and Morel; and a standard localization result \ref{thm:bordloc}
in (complex) cobordism theory. The former concern the algebraic cobordism
ring $\Omega_*(k)$, and the latter concern the complex bordism ring $MU_*$.
These are related via the Lazard ring $\bL_*$; combining
Quillen's theorem \cite[II.8]{Adams} and the Morel-Levine
theorem \cite[4.3.7]{LevineMorel}, we have graded ring isomorphisms:
\[  \Omega_*(k)\cong\bL_*\cong MU_{2*}.  \]

Here is the Levine-Morel generalized degree formula for an 
irreducible projective variety $X$, taken from 
\cite[Theorem 4.4.15]{LevineMorel}. It concerns the
ideal $M(X)$ of $\Omega_*(k)$ generated by the classes $[Z]$ of 
smooth projective varieties $Z$ such that
there is a $k$-morphism $Z\to X$, and $\dim(Z)<\dim(X)$.

\begin{thm}[Generalized Degree Formula]\label{thm:gendegree}
Let $f: Y \to X$ be a morphism of smooth projective $k$-varieties.
If $\dim(X)=\dim(Y)$ then $[Y] - \deg(f)[X]\in M(X)$.
\end{thm}

Trivially, if $[Z]\in M(X)$ then $M(Z)\subseteq M(X)$. We also have:

\begin{lem}\label{lem:Mideal}
Let $X$ be a smooth projective $k$-variety. 
If $Z$ and $Z'$ are birationally equivalent, then
$[Z]\in M(X)$ holds if and only if $[Z']\in M(X)$.
\end{lem}

\begin{proof}
By \cite[4.4.17]{LevineMorel}, the class of $Z$ modulo $M(Z)$ is a
birational invariant. Thus $[Z']-[Z]\in M(Z)$.
Because $M(Z)\subseteq M(X)$, the result follows.
\end{proof}

We shall also need the Levine-Morel ``higher degree formula'' 
\ref{thm:higherdegree}, which is taken from 
\cite[Theorem 4.4.24]{LevineMorel}, and concerns the mod $p$
characteristic numbers $t_{d,r}(X)$ of \cite[4.4]{LevineMorel},
where $p$ is prime, $n\ge1$ and $d=p^n-1$.

Choose a graded ring homomorphism $\psi:\bL_* \to \bF_p[v_n]$ corresponding 
to some height $n$ formal group law, where $v_n$ has degree $d$; 
many such group laws exist,
and the class $t_{d,r}$ will depend on this choice, but only up to a unit.

\begin{defn}\label{def:todd}
For $r>0$, the homomorphism $t_{d,r}:\Omega_{rd}(k)\cong\bL_{rd}\to\bF_p$ 
sends $x$ to the coefficient of $v_n^r$ in $\psi(x)$. If $X$ is a smooth 
projective variety over $k$, of dimension $rd$, then $X$ determines a 
class $[X]$ in $\Omega_{rd}(k)$, and $t_{d,r}(X)$ is $t_{d,r}([X])$.
\end{defn}

\begin{thm}[higher degree formula] 
\label{thm:higherdegree}
Let $f:W \to X$ be a morphism of smooth projective varieties
of dimension $rd$ and suppose that $X$ admits a sequence of surjective
morphisms\vspace{-3pt}
\[ X = X^{(r)} \to X^{(r-1)} \to \dotsm \to X^{(0)} = \Spec(k)\]
such that
\begin{enumerate}
\item $\dim(X^{(i)}) = i\, d.$
\item If $\eta$ is a zero-cycle on 
$X^{(i)}\!\times_{X^{(i-1)}}\!k(X^{(i-1)})$, 
then $p$ divides the degree of $\eta$.
\end{enumerate}
Then $t_{d,r}(W) = \deg(f)\, t_{d,r}(X)$. 
\end{thm}
\goodbreak

Here are some properties of this characteristic number that we shall need.
Recall that 
if $\dim(X)=d$ then $p$ divides $s_d(X)$, so that $s_d(X)/p$ is an integer.

\begin{lem}\label{lem:todd}
Let $X/k$ be a smooth projective variety, and $k\subseteq\C$ and embedding.
\begin{enumerate}
\item For $r=1$, there is a unit $u\in\bF_p$ such that 
$t_{d,1}(X) \equiv u\, s_d(X)/p$.
\item If $X=\prod_{i=1}^r X_i$ and $\dim(X_i)=d$, then 
	$t_{d,r}(X) = \prod_{i=1}^r t_{d,1}(X_i).$
\item $t_{d,r}(X)$ depends only on the class of $(X\times_k \C)^{an}$
in the complex cobordism ring.
\end{enumerate}
\end{lem}

\begin{proof}
Part (1) is \cite[Proposition 4.4.22.]{LevineMorel}. Part (2) is immediate
from the definition of $t_{d,r}$ and the graded multiplicative
structure on $\Omega_*(k)$. Finally, part (3) is a consequence of the
fact that the natural homomorphism $\Omega_*(k) \to MU_{2*}$ is an
isomorphism (since both rings are isomorphic to the Lazard ring).
\end{proof}

\begin{subrem}
The class called $s_d$ in this article is the $S_d$ in \cite{LevineMorel};
the class called $s_d(X)$ in \cite{LevineMorel} is our class $s_d(X)/p$.
\end{subrem}

The next lemma is a variant of Theorem
\ref{thm:higherdegree}. It uses the same hypotheses.

\begin{lem}\label{lem:M(X)}
Let $X$ be as in Theorem \ref{thm:higherdegree}. Then $\psi(M(X)) = 0$.
\end{lem}

\begin{proof}
Consider $Z$ with $[Z]\in M(X)$. If $d$ does not divide $\dim(Z)$, then 
$\psi([Z])= 0$ for degree reasons. If $\dim(Z) = 0$, then the image of $Z$ 
is a closed point of $X$; since the degree of such a closed point is 
divisible by $p$, we have $\psi([Z]) = 0$. Hence we may assume that
$\dim(Z)=sd$ for some $0<s<r$. The cases $r=1$ and $s=0$ are immediate, so we
proceed by induction on $r$ and $s$.

Let $f: Z\to X$ be a $k$-morphism with $\dim(Z)=sd$, and let 
$f_s:Z\to X^{(s)}$ be the obvious composition. As $\dim(Z)=\dim(X^{(s)})$,
the generalized degree formula \ref{thm:gendegree} applies to show
that $[Z] - \deg(f_s)([X^{(s)}])\in M(X^{(s)})$. By induction on $r$, 
$\psi(M(X^{(s)}))=0$, so $\psi([Z]) = \deg(f_s)\psi([X^{(s)}])$. 
We claim that $\deg(f_s)\equiv0\pmod{p}$, 
which yields $\psi([Z])=0$, as desired.

If $f_s$ is not dominant, then $\deg(f_s) = 0$ by definition.
On the other hand, if $f_s$ is dominant, then the generic point of $Z$
maps to a closed point $\eta$ of $X^{(s+1)}\times_{X^{(s)}}k(X^{(s)})$. 
By condition (2) of Theorem \ref{thm:higherdegree}, $p$ divides
$\deg(\eta)=\deg(f_s)$. 
\end{proof}

We will need to show that $\psi(M(Y)) = 0$ for the $Y$ appearing in
Theorem \ref{thm:DN}. This is accomplished in the next lemma.

\begin{lem}\label{lem:M(Y)}
Suppose $X$, $Y$ and $W$ are smooth projective varieties of
dimension $rd$ over $k$,
and $f: W \to X$ and $g: W \to Y$ are morphisms. 
Suppose further that $\psi(M(X)) = 0$ and that $p$ does not
divide $\deg(g)$. Then $\psi(M(Y)) = 0$.
\end{lem}

\begin{proof}
Suppose $[Z]\in M(Y)$. As $g:W\to Y$ is a proper morphism
of smooth varieties, of degree prime to $p$, 
we can lift the generic point $\Spec(k(Z))\to Y$ to a point
$q:\Spec(F)\to W$ for some field extension $F/k(Z)$ of degree $e$ prime
to $p$. Let $\tilde{Z}$ be a smooth projective model of $F$ possessing
a morphism to $Z$ and a morphism to $X$ extending the $k$-morphism
$f\circ q: \Spec(F) \to X$. Hence $[\tilde{Z}]\in M(X)$.
By the degree formula for the map $\tilde{Z} \to Z$, 
$e\,[Z] - [\tilde{Z}]\in M(Z)$. If $\dim(Z) = 0$, then $M(Z) = (0)$. 
In general, $M(Z)$ is generated by the classes of varieties of dimension less than 
$\dim(Z)$ that map to $Z$ (hence a fortiori also map to $Y$) over $k$.
By induction on the dimension of $Z$, we may assume that $\psi(M(Z))=0$. 
Moreover, $\psi([\tilde{Z}]) = 0$ by assumption; since $p$ does
not divide $e$, we conclude that $\psi([Z]) = 0$ as asserted.
\end{proof}

Finally, we will use the following standard bordism localization result.

\begin{lem}\label{Gbord}
Suppose that the abelian $p$-group $G=\mu_p^n$ acts without fixed points on an
almost complex manifold $M$. Then $\psi([M])=0$ in $\bF_p$.
\end{lem}
\begin{proof}
By \cite{tomDieck}, $[M]$ is in the ideal of $MU_*$ generated by 
$\{ p,[M_1]\dots,[M_{n-1}]\}$, where $\dim_{\C}(M_i)=p^i-1$.
Since $p$ is the only generator of this ideal whose dimension is a
multiple of $d=p^n-1$, $\psi$ is zero on every generator 
and hence on the ideal.
\end{proof}

\begin{thm}\label{thm:bordloc}
Let $G$ be $\mu_p^n$ and 
let $X$ and $Y$ be compact complex $G$-manifolds which are 
$G$-fixed point equivalent.
Then $\psi([X])=\psi([Y])$.
\end{thm}

\begin{proof}
Remove equivariantly isomorphic small balls about the fixed points 
of $X$ and $Y$, and let $M=X\cup-Y$ denote the result of joining the rest of
$X$ and $Y$, with the opposite orientation on $Y$. Then $M$ has a
canonical almost complex structure, $G$ acts on $M$ with no fixed points,  
and $[X]-[Y]=[M]$ in $MU_*$.
By Lemma \ref{Gbord}, $\psi([X])-\psi([Y])=\psi([M])=0$.
\end{proof}

We can now prove Theorem \ref{thm:DN}. Note that the inclusion $k(Y)\subset F$
induces a dominant rational map $W\to Y$\!; we may replace $W$ by a
blowup to eliminate the points of indeterminacy and obtain a morphism 
$g:W\to Y$\!, whose degree is prime to $p$,
without affecting the statement of Theorem \ref{thm:DN}.

\begin{proof}[Proof of the DN Theorem \ref{thm:DN}]
We will apply Theorem \ref{thm:higherdegree} to $X$ and the
$X^{(t)} = \prod_{i=1}^t X_i$. We must first check that the
hypotheses are satisfied. The first condition is obvious.
For the second condition, it is convenient to fix $t$ and set
$F=k(X_1\times\cdots\times X_{t-1})$, 
$X'=X^{(t)}\times_{X^{(t-1)}}F$.
By hypotheses (1--2) of Theorem \ref{thm:DN}, 
the symbol $u_t$ is nonzero over $F$ but splits over the generic point 
of $X'$; by specialization, it splits over all closed points. 
A transfer argument implies that the degree of any closed point 
$\eta$ of $X'$ 
is divisible by $p$; this is the second condition.
Hence Theorem \ref{thm:higherdegree} applies and we have
$t_{d,r}(W) = \deg(f)\,t_{d,r}(X)$.

By Lemmas \ref{lem:M(Y)} and \ref{lem:M(X)}, we have that
$\psi(M(Y)) = 0$; by the generalized degree formula
\ref{thm:gendegree}, we conclude that 
$\psi([W])=\deg(g)\,\psi([Y])$, so that $t_{d,r}(W)=\deg(g)\,t_{d,r}(Y)\ne0$.
Hence 
\[ \deg(f)\,t_{d,r}(X) = \deg(g)\,t_{d,r}(Y).  \]
 
By Theorem \ref{thm:bordloc} and Lemma \ref{lem:todd}(3),
$mt_{d,r}(X) = t_{d,r}(Y)$. 
Condition (3) of Theorem \ref{thm:DN} and
Lemma \ref{lem:todd} imply that $t_{d,1}(X_i)\ne0$ for all
$i$ and hence that $t_{d,r}(X) \ne0$. It follows that
$m\deg(g)\equiv \deg(f)\ne0$ modulo $p$, as required.
\end{proof}

\bigskip\bigskip

\subsection*{Acknowledgements}
We would like to state the obvious: we are deeply indebted to Markus Rost
for proving the results presented in this paper, and for lecturing on them
during Spring 2000 and Spring 2005 terms at the Institute for Advanced Study.
We are also grateful to the Institute for Advanced Study for providing the
conditions for these lectures. We are also grateful to Marc Levine and
Peter Landweber for their help with the cobordism theory used in this paper.

\bigskip\goodbreak

\end{document}